\documentclass[11pt]{amsart}
\usepackage{amsmath,amsthm, amscd, amssymb, amsfonts}
\usepackage[all]{xy}
\usepackage{hhline}
\usepackage[dvips, dvipsnames, usenames]{color}

\usepackage[active]{srcltx}

\newcommand{\nc}{\newcommand}
\newcommand{\ot}{\otimes}

\newcommand{\co}{\operatorname{co}}

\newcommand{\ord}{\operatorname{ord}}

\nc{\ydk}{^{K}_{K}\mathcal{YD}}

\newcommand{\gr}{\operatorname{gr}}

\nc{\Tr}{\mathrm{Tr}}
\nc{\X}{\mathbf{X}}

\newcommand{\N}{{\mathbb N}}

\newcommand{\matha}{\mathcal{A}}

\newcommand{\M}{{\mathcal M}}

\newcommand{\D}{{\mathcal D}}
\newcommand{\II}{{\mathcal I}}

\newcommand{\m}{\mathcal{M}}

\nc{\eps}{\varepsilon}

\newcommand{\Oc}{{\mathcal O}}

\newcommand{\ydh}{{}^H_H\mathcal{YD}}

\newcommand\rad{\operatorname{rad}}

\newcommand{\com}{\Delta}

\newcommand\Rep{\operatorname{Rep}}

\newcommand\ad{\operatorname{ad}}

\newcommand{\Comod}{\mbox{\rm Comod\,}}

\nc{\coM}{\M^\ast(2,\Bbbk)}
\nc{\coMtres}{\M^\ast(3,\Bbbk)}
\nc{\coMcua}{\M^\ast(4,\Bbbk)}
\nc{\coMcin}{\M^\ast(5,\Bbbk)}
\nc{\coMt}{\M^\ast(t,\Bbbk)}
\nc{\coMj}{\M^\ast(j,\Bbbk)}
\nc{\coMn}{\M^\ast(n,\Bbbk)}
\nc{\coMd}{\M^\ast(d,\Bbbk)}
\nc{\Ho}{H_0}
\nc{\GH}{G(H)}
\nc{\mas}{\oplus}
\nc{\cA}{\mathcal{A}}
\nc{\yd}{^{C_2}_{C_2}\mathcal{YD}}
\nc{\PH}{\cP(H)}
\nc{\e}{\varepsilon}
\nc{\GL}{\operatorname{GL}}
\nc{\wact}{\rightharpoonup}
\nc{\cark}{char\,k}
\nc{\adl}{\ad_\ell}
\nc{\cP}{\mathcal{P}}
\nc{\cU}{\mathcal{U}}
\nc{\fD}{\mathfrak{D}}
\nc{\cE}{\mathcal{E}}
\nc{\cS}{\mathcal{S}}
\nc{\be}{\textbf{e}}

\newcommand\id{\operatorname{id}}

\newcommand{\Ll}{{\mathcal L}}

\def\pf{\begin{proof}}
\def\epf{\end{proof}}

\theoremstyle{remark}

\numberwithin{equation}{section}
\theoremstyle{plain}

\newtheorem{lema}{Lemma}[section]
\newtheorem{theorem}[lema]{Theorem}
\newtheorem{cor}[lema]{Corollary}

\newtheorem{prop}[lema]{Proposition}

\theoremstyle{definition}

\theoremstyle{remark}
\newtheorem{obs}[lema]{Remark}
\newtheorem{remark}[lema]{Remark}
\newtheorem{rmk}[lema]{Remarks}

\theoremstyle{plain}
\newcounter{maint}

\newtheorem{mainthm}[maint]{Theorem}

\theoremstyle{plain}

\begin{document}

\renewcommand{\baselinestretch}{1.2}

\thispagestyle{empty}

\title[Techniques for classifying Hopf algebras]
{Classifying   Hopf algebras  of a given dimension}
\author[M. Beattie and  G. A. Garc\'\i a]
{Margaret Beattie \and Gast\'on Andr\'es Garc\'\i a}
\thanks{This work was partially supported by
 ANPCyT-Foncyt, CONICET, Ministerio de Ciencia y Tecnolog\'\i a (C\'ordoba)
 and Secyt (UNC)}
\address{\newline \noindent Mount Allison University
\newline \noindent Sackville, NB E4L 1E6
\newline \noindent Canada
\vspace*{0.5cm}
\and
\newline \noindent Facultad de Matem\'atica, Astronom\'\i a y F\'\i sica,
\newline \noindent Universidad Nacional de C\'ordoba. CIEM -- CONICET.
\newline \noindent Medina Allende s/n
\newline \noindent (5000) Ciudad Universitaria, C\'ordoba, Argentina}
\email{mbeattie@mta.ca} \email{ggarcia@famaf.unc.edu.ar}

\date{\today}

\begin{abstract} Classifying all Hopf algebras of a given finite dimension over
$\mathbb{C}$ is a challenging
 problem which remains open even for many small dimensions,
not least because few general approaches to the problem are known.
 Some useful techniques include counting the dimensions of
spaces related to the coradical
 filtration \cite{fukuda-pq}, \cite{andrunatale}, \cite{bitidasca},
studying sub- and quotient Hopf algebras
 \cite{GG}, \cite{GV}, especially those sub-Hopf algebras generated
by a simple subcoalgebra \cite{natale},
 working with the antipode \cite{Ng, Ng2, Ng3, Ng4}, and studying Hopf algebras
in Yetter-Drinfeld categories to help
 to classify Radford biproducts \cite{ChNg}.
In this paper, we add to the classification tools
 in \cite{bg} and apply our results to Hopf algebras of
dimension $rpq$ and $8p$ where $p,q,r$ are distinct  primes.
  At the end of this paper we
summarize in a table the status of the classification for dimensions
up to $100$ to date.
\end{abstract}

\maketitle

\section{Introduction}
Let $\Bbbk$ be an algebraically closed field of characteristic $0$. The
question of classifying all Hopf algebras of a given
dimension over $\Bbbk$ goes back
to Kaplansky in 1975. To date, there are very few general results.
 The Kac-Zhu Theorem \cite{Z}, states that a Hopf
algebra of prime dimension is isomorphic to a group algebra. S.-H. Ng
\cite{Ng} proved that in dimension $p^{2}$, the only Hopf algebras
are the group algebras and the Taft algebras, using previous results
in \cite{andrussch}, \cite{masuoka-p^n}. It is a common belief that
a Hopf algebra of dimension $pq$, where $p$ and $q$ are distinct
prime numbers, is semisimple. Hence, it should be isomorphic to a
group algebra or a dual group algebra by \cite{EG}, \cite{GW},
\cite{ma-2p}, \cite{pqq2}, \cite{So}.
This conjecture has been verified for some particular
values of $p$ and $q$, see \cite{andrunatale, bitidasca,
etinofgelaki2, Ng2, Ng3, Ng4}.
Hilgemann and S.-H. Ng gave the classification of
Hopf algebras of dimension $2p^{2}$ in \cite{hilgemann-ng} and more
recently Cheng and  Ng \cite{ChNg} studied the case $ 4p $, solving
the problem for dimension $20, 28$ and $44$.

In fact, all Hopf algebras of dimension $\leq 23$ are
classified: for dimension $\leq 11$ the problem was solved by
\cite{W}; an alternative proof appears in \cite{stefan}. The
classification for dimension 12 was done by \cite{fukuda} in the
semisimple case and then completed by \cite{natale} in the general
case and for dimension $16$ it was solved by \cite{kashina},
\cite{pointed16}, \cite{biti}, \cite{de1tipo6chevalley} and
\cite{GV}.   For dimension 18 the problem was solved by D. Fukuda
\cite{d-fukuda}  and recently Cheng and Ng finished the
classification for dimension $20$. For the   state of the
classification of low dimensional Hopf algebras as of 2009, see \cite{biti2}.

\smallbreak The classification appears   more difficult for even
dimensions as studied in this article. One reason may be
that  for $H$
  a nonsemisimple Hopf algebra of odd dimension, either $H$ or $H^{*}$ has a
nontrivial grouplike element. The smallest dimension that is still
unclassified is $24$ and, since the classification for
dimension $27$ was recently completed
in \cite{bg}, the next unclassified dimension after $24$ is $32$.

\smallbreak In this paper we study   Hopf algebras over
$\Bbbk$  whose dimension is either smaller than $100$ or
  can be decomposed into the product of a small number of prime numbers.
In particular, we give some partial results on Hopf algebras of
dimension $8p$, with applications to the case of dimension $24$, and
dimension $rpq$, where $r,\ p,\ q$ are distinct  prime numbers.
Since there are many results on the classification problem for
dimension
 $4p$ \cite{ChNg} but the complete classification is incomplete,
we cannot hope to complete the classification
  for dimension $8p$.  However we can narrow the possibilities.

We will say
that a Hopf algebra $H$  is of {\it type}
$(r,s)$
if $|G(H)|=r$ and $|G(H^{*})|=s$.

\begin{mainthm}\label{thm:8p}
   Let $ H $ be a nonsemisimple Hopf algebra of dimension
$ 8p $ with $p$ an odd prime.  If $H$ is not of type $ (r,s) $ with
$ r,s $ powers of $ 2 $, $ (2p,2) $ or $ (2p,4) $, then $H$ is
pointed or basic.
\end{mainthm}

Using counting arguments we can improve the theorem above
in case $ p=3 $.

\begin{mainthm}\label{thm:24}
Let $ H $ be a   Hopf algebra of dimension $ 24$ such that the coradical is not a sub-Hopf algebra
of $H$.
Then $H$ is of type
$ (2,2) $, $ (2,4) $ or $ (6,4) $.
\end{mainthm}

\section*{Acknowledgements}
This version of the paper contains a correction on the published version. The statement 
and proof of Proposition \ref{prop:R-skew} are changed and the proof of the
results that follow from it are corrected accordingly. 
We thank H.-S. Ng for kindly communicating the gap to us and for
the careful reading of our paper.

\section{Preliminaries}
In this section we introduce notation, recall some previous results
which help with the classification of finite dimensional Hopf
algebras, see \cite{andrunatale}, \cite{natale}, \cite{stefan},
\cite{bitidasca}, \cite{biti}, \cite{GV}, \cite{d-fukuda}, and
introduce a few new ones. For the general theory of Hopf algebras
see \cite{Mo}, \cite{S}.

\subsection{Conventions}\label{subsec:conv}
Throughout this paper $ p,q $ will denote   odd prime numbers, $
C_{k} $ the cyclic group of order $ k $ and $\mathbb{D}_{k}$ the
dihedral group of order $2k$. Unless otherwise specified, all Hopf
algebras in this article are finite dimensional over   a field
$\Bbbk$ algebraically closed  of characteristic zero.

\begin{remark}\label{rm: LR2} By \cite{LR2}, with the assumptions
above,
 a Hopf algebra is semisimple if and only if
it is cosemisimple if and only if $S^2$, the square of the antipode,
is the identity. Thus if $L$,$K$ are semisimple sub-Hopf algebras of
a Hopf algebra $H$, then $\langle L,K \rangle$, the sub-Hopf algebra
of $H$ generated by $L$ and $K$ is semisimple since $S_H^2$ is the
identity on $L$ and on $K$.
\end{remark}

For $H$  a Hopf algebra over $\Bbbk$ then $\com$, $\e$, $S$
denote respectively the comultiplication, the counit and the
antipode; $\GH$ denotes the group of grouplike elements of $H$; $H_{0}$
denotes the coradical;
$(H_n)_{n \in \N}$ denotes the coradical filtration of $H$ and $L_h$
(resp. $R_h$) is the left (resp. right) multiplication in $H$ by
$h$. We say that $H$ is \textit{pointed} if $H_{0} = \Bbbk G(H)$.

The set of $(h,g)${\it -primitives} (with
$h,g\in\GH$) and set of {\it skew-primitives} of $H$ are:
$$
\begin{array}{rcl}
\cP_{h,g}(H)&:=&\{x\in H\mid\com(x)=x\ot h+g\ot x\},\\
\noalign{\smallskip} \cP(H)&:=&\sum_{h,g\in\GH}\cP_{h,g}(H).
\end{array}
$$
We say that $x\in \Bbbk(h-g)$ is a {\it trivial} skew-primitive; a
skew-primitive not contained in $\Bbbk G(H)$ is \textit{
nontrivial}.

Let $\coMn$ denote the simple coalgebra of
dimension $n^2$, dual to the matrix algebra ${\mathcal M}(n,\Bbbk)$.
We say that a coalgebra $C$ is a $d\times d$ \textit{matrix-like
coalgebra} if $C$ is spanned by elements $(e_{ij})_{1\leq i,j\leq n}$
such that $\com(e_{ij}) =
\sum_{1\leq l \leq n} e_{il} \ot e_{lj}$ and $\e(e_{ij}) = \delta_{ij}$.
If  the set $(e_{ij})_{1\leq i,j\leq d}$ of
a coalgebra $C$ of dimension $d^{2}$ is linearly independent,
following \c Stefan we call $\be =  \{e_{ij}:\ 1\leq i,j\leq d\}$
a \textit{multiplicative matrix} and then
 $ C\simeq \coMd $ as coalgebras.

Since the only semisimple and pointed Hopf algebras are the group
algebras, we shall adopt the convention that `pointed' means
`pointed nonsemisimple'.
Similarly, we say that a finite
dimensional Hopf algebra $H$ is \textit{basic} if $H$ is
basic as an algebra or copointed as a coalgebra, i.e.,
all simple $H$-modules are one-dimensional or the dual is pointed,
but $H$ is not the dual of a group algebra.

\smallbreak Recall that a tensor category $\mathcal{C}$  over $\Bbbk$
has the Chevalley property if the tensor product of any two simple
objects is semisimple. We shall say that a Hopf algebra $H$ has
the \emph{Chevalley property} if the category $\Comod (H)$ of
$H$-comodules does.

\begin{rmk}
(i) The notion of the Chevalley property in the setting of Hopf
algebras was introduced by \cite{AEG}: it is said in \textit{loc.
cit.} that a Hopf algebra has the Chevalley property if the category
$\Rep(H)$ of $H$-modules does.

\smallbreak (ii) Unlike \cite{AEG}, in \cite[Section
1]{de1tipo6chevalley}, the authors refer
to the Chevalley property in the
category of $H$-comodules; this definition is the one we adopt.
Note that it is equivalent to say that the coradical $H_{0}$ of $H$ is a
sub-Hopf algebra.

\smallbreak (iii) If $H$ is semisimple or pointed then it has the
Chevalley property.
\end{rmk}

\par Let $N$ be a positive integer and let
$q$ be a primitive $ N^{th} $  root of unity. We denote by $ T_{q} $
the Taft algebra which is generated as an algebra by the elements $
g $ and $ x $ satisfying the relations $ x^{N} = 0 = 1-g^{N}$, $
gx=q xg $. Taft Hopf algebras are self-dual and pointed of dimension
$ N^{2} $  with $g$   grouplike and
$x$ a $(1,g)$-primitive, i.e.,  $ \com(g) = g\ot g $ and $
\com(x) = x\ot 1 + g\ot x $.  If $N=2$ so that $q = -1$, then $T_{-1}$
is called the Sweedler Hopf algebra
and will be denoted $H_4$ thoughout this article.

\subsection{Spaces of coinvariants}\label{sec: coinvariant}
Let $K$ be a coalgebra with a distinguished grouplike 1. If $M$
is a right $K$-comodule via $\delta$, then the space of {\it right
coinvariants} is $$ M^{\co \delta} = \{x\in
M\mid\delta(x)=x\ot1\}.
$$
 Left
coinvariants are defined analogously. If $\pi:H\rightarrow K$ is a morphism of Hopf
algebras, then $H$ is a right $K$-comodule via $(1\ot\pi)\com$.
In this case $H^{\co \pi}:=H^{\co (1\ot\pi)\com}$ and
$H^{co\pi}$ is a subalgebra of $H$.

We make the following observation.

\begin{lema}\label{lm: on R pi=epsilon}
Let $\pi: H \rightarrow K$ be a Hopf algebra map and let $R:= H^{co\pi}$. Then
 $\pi|_R = \varepsilon|_R$.
\end{lema}

\begin{proof} Let $z \in H^{co \pi}$.  Since $\pi$ is a morphism of Hopf algebras,
\begin{displaymath}
\Delta_{K} \pi(z) = (\pi \otimes \pi) \Delta_H(z) = \pi(z_1) \otimes \pi(z_2) = \pi(z) \otimes 1,
\end{displaymath}
so that, applying $m_K \circ (\varepsilon_K \otimes \id_{K})$ to the equation above,
 we obtain $\pi(z) = \varepsilon_K(\pi(z))\in \Bbbk$. Again, since $\pi$ is a Hopf algebra map,
 $\varepsilon_K(\pi(z)) = \varepsilon_H(z)$ and so $\pi(z) = \varepsilon_H(z)$.
\end{proof}

\subsection{Extensions of Hopf algebras}\label{subsec:extensions}
Recall \cite{andrudevoto} that an exact sequences of Hopf algebras is  a sequence of Hopf algebra
morphisms  $A\overset{\imath}\hookrightarrow
H\overset{\pi}\twoheadrightarrow B$ where $A,H,B$ are any Hopf algebras, $\imath$ is injective, $\pi$ is
surjective, $\pi \imath = \varepsilon_A$, $\ker \pi = A^+H$ and $ A = H^{co \pi}$.
 An exact sequence is called {\it central}
if $A$ is contained in the centre of $H$.

The next result will be useful throughout. For a proof see \cite[Lemma 2.3]{GV}.

\begin{lema}  \label{lm: dim H co dim B = dim H}
If $\pi:H\rightarrow B$ is an epimorphism of Hopf algebras
  then $\dim H=\dim H^{\co \pi}\dim B$. Moreover, if
$A=H^{\co \pi}$ is a sub-Hopf algebra of $H$ then the sequence $
A\overset{\imath}\hookrightarrow H\overset{\pi}\twoheadrightarrow B
$ is exact. \qed
\end{lema}

The following proposition tells us how to
compute,  in a particular case,
the dimension of the coradical of  $ H^{*} $
using exact sequences.

\begin{prop}\label{prop:exact-dim-cor}
Let $\Gamma$ be a finite group and
$A \hookrightarrow H \twoheadrightarrow \Bbbk \Gamma$
an exact sequence of
 Hopf algebras.
Then $\dim (H^{*})_{0} = \dim (H/\rad H) =
|\Gamma|\dim (A^{*})_{0}$.
\end{prop}

\begin{proof}
The statement follows from the proof of \cite[Lemma 5.9]{GV}. The
idea is the following: since the sequence is exact, $H$ is the
$\Gamma$-crossed product $A*\Gamma$. Let $g \in \Gamma$, then the
weak action of $g$ on $A$ defines an algebra map and consequently
$\rad A$ is stable by $\Gamma$. Then $\rad A*\Gamma$ is a nilpotent
ideal of $A*\Gamma$ and $\rad A*\Gamma\subseteq \rad H$. Since
$H/(\rad A*\Gamma)$ is semisimple, it follows that $\rad H \subseteq
\rad A*\Gamma$ and hence $\dim  (H^{*})_{0} = \dim (H/\rad H) = \dim
( A*\Gamma / \rad A*\Gamma)= |\Gamma|\dim (A/\rad A) =|\Gamma|\dim
(A^{*})_{0}  $.
\end{proof}

\subsection{Yetter-Drinfel'd modules} \label{subsect:yd} For $H$
any
Hopf algebra, a left Yetter-Drinfeld module $ M $ over $H$ is a left
$H$-module $(M,\cdot)$ and a left $H$-comodule $(M,\delta)$ such
that for all $h \in H, m \in M$,
$$
\delta(h \cdot m) = h_{{1}} m_{(-1)}\cS(h_{{3}}) \otimes
h_{{2}}\cdot m_{(0)},
$$
where $\delta(m) = m_{(-1)}\otimes m_{(0)}$. We will denote this category by
$ \ydh $.

\subsection{On the coradical filtration}\label{subsec:coradical filtration}
We begin by recalling a description of the coradical filtration due
to Nichols. More detail can be found in \cite[Section
1]{andrunatale}.

Let $D$ be a coalgebra over $\Bbbk$.
Then there exists a coalgebra projection
$\pi:  D \to D_{0 }$ from
$ D$ to the coradical
$ D_{0}$ with kernel $I$, see
\cite[5.4.2]{Mo}. Define the maps
$$\rho_{L}:= (\pi\ot \id)\com: D \to D_{0}\ot D \qquad\mbox{ and }\qquad
\rho_{R}:= (\id\ot \pi)\com: D \to D\ot D_{0},$$
and let $P_{n}$ be the sequence of subspaces defined recursively
by
\begin{align*}
P_{0} & = 0,\\
P_{1} & = \{x\in D:\ \com(x) = \rho_L(x) +\rho_R(x)\}
= \com^{-1}(D_{0}\ot I + I\ot D_{0}),\\
P_{n} & = \{x\in D:\ \com(x) - \rho_L(x) - \rho_R(x) \in
\sum_{1\leq i \leq n-1}P_{i}\ot P_{n-i}\}, \quad n\geq 2.
\end{align*}

Then by a result of Nichols, $P_{n} = D_{n}\cap I$ for $n\geq 0$,
see \cite[Lemma 1.1]{andrunatale}. Suppose that
$D_{0} = \bigoplus_{\tau \in \II} D_{\tau}$, where the $D_{\tau}$ are simple
coalgebras of dimension $d^{2}_{\tau}$.
Any $D_{0}$-bicomodule is a direct sum of simple $D_{0}$-sub-bicomodules
and every simple $D_{0}$-bicomodule has coefficient coalgebras
$D_{\tau}, D_{\gamma}$ and has dimension
$d_{\tau}d_{\gamma} = \sqrt{\dim D_{\tau}\dim D_{\gamma}}$
for some $\tau, \gamma \in \II$, where $d_{\tau},d_{\gamma}$
are the dimensions of the associated comodules of $D_{\tau}$ and
$D_{\gamma}$, respectively.

\smallbreak Now suppose $H$ is a   Hopf algebra.   Then $H_{n},P_{n}$ are
$H_{0}$-sub-bicomodules of $H$ via $\rho_R$ and $\rho_L$. As in
\cite{andrunatale}, \cite{d-fukuda}, for all $n\geq 1$ we denote by
$P_{n}^{\tau,\gamma}$ the isotypic component of the
$H_{0}$-bicomodule of $P_{n}$ of type the simple bicomodule with
coalgebra of coefficients $D_{\tau}\ot D_{\gamma}$. If $D_{\tau} =
\Bbbk g$ for $g$ a grouplike, we use the superscript $g$
instead of $\tau$.   If the simple subcoalgebras are $S(D_\tau)$, $S(D_\gamma)$,
(respectively $gD_\tau$, $D_\tau g$ for
$g$ grouplike)
we write $P_n^{S\tau, S\gamma}$,(respectively $P_n^{g\tau, g\gamma}$,$P_n^{\tau g, \gamma g}$.)
For $D_{\tau}, D_{\gamma}$ simple coalgebras we
denote $P^{\tau,\gamma} = \sum_{n \geq 0}P_{n}^{\tau,\gamma}$.

Similarly, for $\Gamma$ a set of grouplikes of   $H$,
let $ P^{\Gamma,\Gamma}$ denote  $\sum_{g,h \in \Gamma}P^{g,h}
 $
  and let $H^{\Gamma, \Gamma}:= P^{\Gamma, \Gamma}
\oplus \Bbbk\Gamma$. If $\mathcal{D},\mathcal{E}$ are
   sets of simple subcoalgebras, let $P^{\mathcal{D}, \mathcal{E}}$ denote
   $\sum_{D \in \mathcal{D}, E \in \mathcal{E}}P^{D,E}$.
Since $H_n = H_0 \oplus P_n$, we have that $H = H_0 \oplus \sum_{\tau,\gamma}
P^{\tau,\gamma}$.

Following D. Fukuda, we
say that the subspace $P_{n}^{\tau,\gamma}$ is
\textit{nondegenerate} if $P_{n}^{\tau,\gamma} \nsubseteq P_{n-1}$.
The following  results are due to D. Fukuda; note that (ii)  is a
generalization of \cite[Cor. 1.3]{andrunatale} for $n>1$.

\begin{lema}\label{lema:fukuda-deg} \label{lema:fukuda}
\label{lema:fukuda-deg-m}
\emph{(i)} \cite[Lemma 3.2]{fukuda-pq} If the subspace $P_{n}^{\tau,\gamma}$
is nondegenerate for some $n > 1$, then there exists a set of
simple coalgebras $
\{D_{1},\cdots ,D_{n-1} \}$ with
$P_{i}^{\tau,D_{i}}$, $P_{n-i}^{D_{i},\gamma}$   nondegenerate
for all $1\leq i\leq n$.

\par \emph{(ii)} \cite[Lemma 3.5]{fukuda-pq}
 For $S$ the antipode in the Hopf algebra $H$ and $g \in \GH$,
$$\dim P_{n}^{\tau, \gamma} = \dim P_{n}^{S\gamma,S\tau}
= \dim P_{n}^{g\tau, g\gamma}= \dim P_{n}^{\tau g, \gamma g}.$$

\par \emph{(iii)} \cite[Lemma 3.8]{fukuda-pq} Let $C,D$ be
simple subcoalgebras such that $P_{m}^{C,D}$ is nondegenerate. If
$\dim C \neq \dim D$ or $\dim P_{m}^{C,D} - P_{m-1}^{C,D} \neq \dim
C$ then there exists a simple subcoalgebra $E$ such that
$P_{\ell}^{C,E}$ is nondegenerate for some $\ell\geq m+1$.\qed
\end{lema}

The following facts about dimensions from  \cite{andrunatale} will
be useful later.

\begin{lema}\label{lema:andrunatale}\cite{andrunatale} Let $H$ be a Hopf algebra
with $G := G(H)$. Then for $n \geq 0$, $d \geq 1$, $|G|$ divides $\dim H_n$  and
$ \dim H_{0,d}$, where $H_{0,d}$   denotes the
direct sum of the simple subcoalgebras of $H$ of dimension $d^{2}$. Also
$H_n = H_0 \oplus P_n$ so that $|G|$ divides $\dim P_n$.\qed
\end{lema}

It is well-known (see for example \cite{andrunatale}) that
if a Hopf algebra $H$ has a nontrivial skew primitive element, then
$\dim H$ must be divisible by a square.
More precisely we have the following lemma.

\begin{lema}\label{lema: dim H mn rel pr}
Let $H$ be a Hopf algebra with $|G(H)| = m$
and $\dim H = mn$ where $m,n$ are relatively prime.
Then $H$ has no nontrivial
skew-primitive element.
\end{lema}

\begin{proof}
Suppose that $x$ is a nontrivial skew-primitive
element in $H$ and let $L$ be the
sub-Hopf algebra of $H$ generated by $x$ and $G(H)$.
By \cite[5.5.1]{Mo}, $L$ is
pointed. By \cite[Proposition 1.8]{andrunatale},
$\dim L$ is divisible by $rm$ where $r\neq 1$ is a positive integer
dividing $m$. Then $\dim H = mn$ is divisible by $rm$,
contradicting the fact that $(m,n) = 1$.
\end{proof}

The next proposition generalizes results of Beattie and
D\v{a}sc\v{a}lescu \cite{bitidasca} and gives
a lower bound for the
dimension of a finite dimensional Hopf algebra
without nontrivial skew-primitive elements.

\begin{prop}\label{prop:biti-dasca}
\label{cor:bitidasca-p1}\cite[Proposition 3.2]{bg}
Let $H$ be a non-cosemisimple Hopf
algebra with no nontrivial skew-primitives.
\par \emph{(i)} For any $g \in G(H)$ there exists a simple subcoalgebra $C$ of
$H$ of dimension $> 1$ such that $P_{1}^{C,g}\neq 0$, $P_k^{C,D}$ is
nondegenerate for some $k>1$ and $D$ a simple subcoalgebra of the same dimension as
$C$, and $P_m^{g,h}$ is nondegenerate for some $m>1$ and $h$
grouplike.
\par \emph{(ii)}
Suppose $H_0\simeq \Bbbk G \oplus \sum_{i=1}^t{\mathcal M}^*(n_i,\Bbbk)$ with $t\geq 1$,
$2  \leq n_1 \leq \ldots \leq n_t$. Then
\begin{equation}  \label{form: bdf bound}\nonumber
\dim H \geq \dim(H_0) + (2n_1 + 1)|G| + n_1^2.
\end{equation}\qed
\end{prop}

\subsection{Matrix-like coalgebras}
The next theorem due to \c Stefan has been a key component for
several classification results.

\begin{theorem}\label{thm:stefan}\cite[Thm. 1.4]{stefan}
Let $D$ be the simple coalgebra $\mathcal{M}^\ast(2, \Bbbk)$.
\begin{enumerate}
 \item[(i)] For $f$   an antiautomorphism of $D$ such
that $\ord ( f^{2}) = n < \infty$ and $n > 1$,   there exist a multiplicative
matrix $\be$ in
$D$ and a root of unity $\omega$ of order $n$ such that
$$f(e_{12}) = \omega^{- 1} e_{12},\quad f(e_{21}) = \omega e_{21},\quad
f(e_{11}) = e_{22},\quad  f(e_{22}) = e_{11}.$$

\item[(ii)] For $f$ be an automorphism of $D$ of finite order $n$,
  there exist a multiplicative matrix $\be$ on $D$
and a root of unity $\omega$ of order $n$ such
that $f(e_{ij}) = \omega^{i-j} e_{ij}$.
 \qed
\end{enumerate}
 \end{theorem}

Now we recall some useful results on matrix-like coalgebras.
In \cite{bitidasca}  all $2\times 2$ matrix-like
coalgebras of dimension less than $4$ were described; we summarize in the
following theorem.

\begin{theorem}\cite[Thm. 2.1]{bitidasca}\label{thm:2x2-matrix}
Let $D$ be a $2\times 2$ matrix-like coalgebra of dimension less
than $4$. If $\dim D = 1,2$ then $D$ has a basis of grouplike elements.
If $\dim D = 3$, then $D$ has a basis $\{ g,h,x  \}$ where $g,h$ are grouplike and
$x$ is $(g,h)$-primitive.\qed
\end{theorem}

We end this section with the following lemma.

\begin{lema}\label{lema:proj-coalg}
Let $\pi: H \to H_{4}$ be a Hopf algebra epimorphism.
If $H$ is generated by a simple
subcoalgebra $D$
of dimension $4$, then $\dim\ ^{\co \pi}D \geq 2$ or
$\dim D^{\co \pi} \geq 2$.
\end{lema}

\pf
Since $H_{4}$ is pointed, $\dim \pi(D) < 4$. Moreover, since
$D$ generates $H$ as an algebra, $\pi(D) $ generates $H_4 $.
Then by Theorem \ref{thm:2x2-matrix}, $ \dim \pi(D) =3 $.
  Let $\{e_{ij}\}_{1\leq i,j\leq 2}$
be a multiplicative matrix of $D$, then
$\{\pi(e_{ij})\}_{1\leq i,j\leq 2}$ is a linearly dependent set.

\par As in the proof of Theorem \ref{thm:2x2-matrix},
see \cite[Thm. 2.1]{bitidasca}, we divide the proof into two cases.

\noindent {\bf Case 1:}  The set $\{\pi(e_{11}), \pi(e_{12}),
\pi(e_{21})\}$ is linearly independent.

\par
Then $\pi(e_{22}) = \pi(e_{11}) +
a \pi(e_{12}) + b \pi(e_{21})$ for scalars $a,b$.
By comparing $\Delta \pi(e_{22})$ and $\pi \otimes \pi \circ \Delta (e_{22})$ one sees that
$ab = -1$ so that   there exists $0 \neq b\in \Bbbk$
such that
$\pi(e_{22}) = \pi(e_{11}) + b \pi(e_{12}) -
\frac{1}{b}\pi(e_{21})$.  Then it is straightforward to verify
that
the linearly independent elements $h_{1}= \pi(e_{11}) -
\frac{1}{b}\pi(e_{21})$ and
$h_{2} = \pi(e_{11}) + b\pi(e_{12})$ are grouplike.

\par
Suppose that $h_{1} = 1$. Then it is straightforward to show that $  t_{1}=
e_{11} - \frac{1}{b}e_{21} \in \ ^{\co \pi}D$.

Let $s_{1}= e_{22} - be_{12}  $ and
note that $s_1,t_1$ are linearly independent.
 Then   $\pi(s_{1}) =
\pi(e_{22}) - b\pi(e_{12}) =
\pi(e_{11}) + b \pi(e_{12}) - \frac{1}{b}\pi(e_{21})- b\pi(e_{12})=
\pi(e_{11}) - \frac{1}{b}\pi(e_{21}) = \pi(t_{1}) = 1$.
Then a computation similar to that for $t_1$,
shows that  $s_{1} \in \ ^{\co \pi}D$. Thus $\dim \ ^{\co \pi}D \geq 2$.

 \par
 If $h_2 = 1$ then define $t_2 = e_{11} + be_{12}$ and
$s_2 = e_{22} + \frac{1}{b}e_{21}$.  A computation
 similar to that above shows that $t_2,s_2 \in D^{co\pi}$
so that $D^{co\pi}$ has dimension at least $2$.

\noindent {\bf Case 2:}
The set $ \{\pi(e_{11}),\pi(e_{12}), \pi(e_{22} )\} $
is linearly independent.

Then there exist  $a, b \in \Bbbk$ such
that $\pi(e_{21}) = a\pi(e_{11}) + b\pi(e_{12})- a\pi(e_{22})$.
If $a \neq 0$, then the case reduces to Case 1. If $a = 0$, then
$\pi(e_{21}) = b\pi(e_{12})$ and by comparing
$\Delta \pi (e_{21})$ and $\Delta \pi (be_{12})$, we see that
$b=0$ so that $\pi(e_{21})=0$.
Thus
$ \com(\pi(e_{11})) = \pi(e_{11})\ot \pi(e_{11}) $,
$ \com(\pi(e_{22})) = \pi(e_{22})\ot \pi(e_{22}) $ and
$ \com(\pi(e_{12})) = \pi(e_{11})\ot \pi(e_{12}) +
\pi(e_{12})\ot \pi(e_{22}) $, which implies
that $G(H_{4}) = \langle \pi(e_{11}), \pi(e_{22})\rangle$.
If $ \pi(e_{11}) = 1 $, then $e_{11} \in D^{\co \pi}$,
for
\begin{align*}
(1\ot \pi) \com(e_{11})  =
(1\ot \pi) (e_{11}\ot e_{11} + e_{12}\ot e_{21})
= e_{11}\ot \pi(e_{11}) + e_{12}\ot \pi(e_{21}) =
e_{11}\ot 1.
\end{align*}
Moreover, the element $e_{11} + e_{21}   \in D^{\co \pi}$,
since
\begin{align*}
(1\ot \pi) \com(e_{11} + e_{21}) & =
(1\ot \pi) (e_{11}\ot e_{11} + e_{12}\ot e_{21} +
e_{21}\ot e_{11} + e_{22}\ot e_{21}\\
& = e_{11}\ot \pi(e_{11}) + e_{12}\ot \pi(e_{21})
+e_{21}\ot \pi(e_{11}) + e_{22}\ot \pi(e_{21})  =
e_{11}\ot 1 + e_{21}\ot 1.
\end{align*}
Thus, $\dim D^{\co \pi} \geq 2$. The case $ \pi(e_{22})=1 $
is completely analogous and implies that $\dim \ ^{\co \pi}D\geq 2$,
  taking the elements $e_{22}$ and $e_{22} + e_{21}$.
\epf

Note that if $D$ in Lemma \ref{lema:proj-coalg}
is stable under $S^2$ then the proof can be simplified considerably.
For then we may choose a multiplicative matrix for
$D$ consisting of eigenvectors for $S^2$ and we   must have
that $\pi(e_{ii}) \in \Bbbk C_2 \subset H_4$.

\subsection{Hopf algebras generated by a simple subcoalgebra}
\label{subsec:gen-simple}
In this subsection we summarize some known facts about
Hopf algebras generated by simple subcoalgebras of dimension $4$.
The most important  is the next proposition,
 due to Natale but  derived from
a result of \c{S}tefan \cite{stefan}.

\begin{prop}\label{prop:natale-stefan}\cite[Prop. 1.3]{natale}. Suppose that $H$ is
nonsemisimple Hopf algebra   generated by a simple subcoalgebra of
dimension $4$ which is stable by the antipode. Then $H$ fits into a
central exact sequence $\Bbbk^G\overset{\imath}\hookrightarrow
H\overset{\pi}\twoheadrightarrow A,$ where $G$ is a finite group and
$A^*$ is a pointed nonsemisimple Hopf algebra.\qed
\end{prop}

We have the following useful results from \cite{GV}. If $H$ is a Hopf algebra,
${\mathcal
Z}(H)$ denotes its centre.

\begin{lema}\label{truco util}\cite[Lemma 4.2]{GV}
Let $\pi:H\rightarrow K$ be a morphism of
 Hopf
algebras such that $\pi(g)=1$ for some $g\in\GH,\,g\neq1$. Suppose
that $H$ is generated by a simple
subcoalgebra of dimension $4$ stable by $L_g$. Then $\pi(H)\subseteq
\Bbbk G(K)$. \smallbreak The same holds   with $R_g$ instead of
$L_g$; or with $\adl(g)$ or $\ad_r(g)$ if $g\notin{\mathcal
Z}(H)$.\qed
\end{lema}

\begin{lema}\label{truco util bis}\cite[Lemma 4.3]{GV}
Let $\pi:H\rightarrow K$ be an epimorphism of
Hopf algebras and assume that $K$ is nonsemisimple. Suppose that
$H$ is generated by a
simple subcoalgebra of $H$ of dimension $4$ stable by $S_H^2$.
Then $\ord S_H^2 = \ord S^2_K$.\qed
\end{lema}

%

\begin{obs}
$(i)$ If $H$ is generated as an algebra  by $C\oplus D$ with $S(C) = D$,
then $C$ also generates $H$ as an algebra, since the
sub-bialgebra generated by $C$ is finite dimensional and thus a
sub-Hopf algebra.

\par $(ii)$ Suppose that
$H$ is generated by a
simple subcoalgebra $C$ of dimension $4$ stable by $S_H^2$.
If $H_{4}$ is a sub-Hopf algebra of $H^{*}$, then
$S_{H}^{4} = \id$. For the inclusion $H_{4} \hookrightarrow H^{*}$
induces a Hopf algebra surjection $H\twoheadrightarrow H_{4}$ and by
Lemma \ref{truco util bis} the claim follows.
\end{obs}

We end this section with the following proposition.

\begin{prop}\label{prop:R-skew}\label{pr: HNg 1.3}
Let $ \pi: H \twoheadrightarrow A $ be a Hopf algebra
epimorphism and assume $ \dim H = 2\dim A $. Then
$H^{\co \pi}= \Bbbk\{1,x\}$ with $x$ a $(1,g)$-primitive
element with $\pi(x)=0$, $g\in G(H)$ and  $\ord g = 2n$ with $n\geq 1$.
If $x$ is trivial,  i.e.,  $x \in \Bbbk G(H)$,
 then $H$ fits into an exact sequence of Hopf algebras
$\Bbbk C_{2} \hookrightarrow H \twoheadrightarrow A$.   Otherwise
  $x^{2}=0$,   $xg = -gx$ and $x,g$ generate a pointed sub-Hopf algebra of dimension
$4n$; in particular, $4| \dim H$.
\end{prop}

\begin{proof} The statement follows from the proof of
\cite[Prop. 1.3]{hilgemann-ng}. We reproduce
  part of the proof. Let $ R=H^{\co \pi} $;
it is known that $R$ is a left coideal subalgebra and
$\dim R=2$. Let $x \in R\smallsetminus \{0\}$ such
that $\varepsilon (x) = 0 $. Then $R = \Bbbk\{1,x\}$ and
$\com(x) = a\ot 1 + b\ot x$ for some $a,b\in H$.
Since $x \in R$, it follows that $x = a$ and $\pi(x)=0$.    The
coassociativity of $ \com $ implies that $b$ is grouplike. Denote $g=b$,
then $x\in P_{1,g}$ and $x$ is a skew-primitive element.
If $x=\alpha(1-g)$ for some $\alpha \in \Bbbk^*$, then $g^{2}=1$ and
$R\simeq \Bbbk C_{2}$ implying that $H$ fits into an exact sequence
of Hopf algebras $\Bbbk C_{2} \hookrightarrow H \twoheadrightarrow A$.

 Assume $x$ is non-trivial and let $m=\ord g$.
As $R$ is a subalgebra
stable by the adjoint action of $g$, it follows that $gxg^{-1} = \beta x$ with
$\beta^{m}=1$ and
$\com (x^{2}) = x^{2}\ot 1 + (1+\beta)xg\ot x + g^{2}\ot x^{2}$.
Let $x^2 = \alpha 1 + \gamma x$. Then  applying $\varepsilon$, we see that $\alpha = 0$
 and $x^2 = \gamma x$ so that
$\com(x^2) = \gamma x \otimes 1 + \gamma g \otimes x$. Thus $\beta = -1$ and
 $\gamma g \otimes x =   g^2 \otimes \gamma x $ so that $\gamma =0$.
 Since $\beta = -1$,    $2|m$ and
the subalgebra generated by $g$ and $x$ is a pointed sub-Hopf algebra
of $H$ of dimension $2m = 4n$ for some $n\geq 1$.
\end{proof}


%

\subsection{Hopf algebras of dimension $4p$}\label{sec: 4p general}
This section contains a brief overview of what is known for
dimension $4p$. Knowledge of the classification in this dimension is
of course necessary to understand dimension $8p$ which we study in
the last section of this note. Recall that for dimension $12$ the
classification is due to
\cite{andrunatale}, \cite{fukuda}, \cite{natale}.

\par  Up to isomorphism, the semisimple Hopf algebras of dimension
$4p$ consist of group algebras and their duals, and also of two
  self-dual Hopf algebras, constructed by Gelaki in
\cite{G}, which we will denote by $\mathcal{G}_{1}$ and
$\mathcal{G}_{2}$. Both have group of grouplikes of order $4$ with
$G(\mathcal{G}_{1}) \cong C_4$ and $G(\mathcal{G}_{2}) \cong C_2
\times C_2$.

\subsubsection{Nonsemisimple pointed Hopf algebras of dimension
$ 4p $.}\label{sect: 4p} All pointed Hopf algebras of dimension $4p$
have group of grouplikes isomorphic to $C_{2p}$ and are described in
 \cite[A.1]{andrunatale}.

  \par In particular, let $\matha$ be a pointed Hopf algebra of dimension $ 4p $.
Then, with $g$ denoting a generator of $C_{2p}$, and $\xi$ a primitive $p$th root of unity,
   $\matha$ is isomorphic to exactly one of the following.

   $$
\begin{array}{ll}
\matha(-1,0):=&\Bbbk\langle g,x  \mid
g^{2p}-1=x^2 =gx+xg=  0\rangle,
\\ \noalign{\smallskip}
&\com(g)=g\ot g,\quad\com(x)=x\ot 1  + g\ot x.
\end{array}
$$

 $$
\begin{array}{ll}
\matha(-1,0)^\ast:=&\Bbbk\langle g,x  \mid
g^{2p}-1=x^2 =gx+ \xi xg=  0\rangle,
\\ \noalign{\smallskip}
&\com(g)=g\ot g,\quad\com(x)=x\ot 1+g^p\ot x .
\end{array}
$$

 $$
\begin{array}{ll}
\matha(-1,1) :=&\Bbbk\langle g,x  \mid
g^{2p}-1= x^2 - g^2 + 1 =gx+   xg=  0\rangle,
\\ \noalign{\smallskip}
&\com(g)=g\ot g,\quad\com(x)=x\ot 1+g \ot x .
\end{array}
$$

$$
\begin{array}{ll}
H_{4} \otimes \Bbbk C_p:=&\Bbbk\langle g,x  \mid
g^{2p}-1= x^2   =gx+   xg=  0\rangle,
\\ \noalign{\smallskip}
&\com(g)=g\ot g,\quad\com(x)=x\ot 1+g^p \ot x .
\end{array}
$$

Note that $H_4 \otimes \Bbbk C_p$ is self-dual.  The Hopf algebra $\matha(-1,1)$ is a nontrivial lifting
of $\matha(-1,0)$ and has nonpointed dual.  The nonpointed Hopf algebra $\matha(-1,1)^\ast$ contains a copy of the Sweedler Hopf algebra and
as a coalgebra $ \matha(-1,1)^\ast \cong H_4 \otimes \mathcal{M}^\ast(2, \Bbbk)^{p-1}$.
The Hopf algebras
$\matha(-1,0)$ and $\matha(-1,1)$ do not have sub-Hopf algebras isomorphic to $H_4$ but
$\matha(-1,0)^\ast$ and $H_4 \otimes \Bbbk C_p$ do.
  In all four cases, $S^4 = \id$. In Section \ref{sec:8p} we
will use this notation for these pointed Hopf algebras. \vspace{1mm}

\subsubsection{Nonsemisimple nonpointed Hopf algebras of dimension $ 4p$}
These Hopf algebras  were studied
  in \cite{ChNg} with the classification completed  for $p=3, 5, 7,
  11$. The main theorems of \cite{ChNg} are:

  \begin{theorem} \label{th: ChNg I} \cite[Theorem I]{ChNg} For $H$ a nonsemisimple Hopf
  algebra of dimension $4p$, then $H$ is pointed if and only if
  $|G(H)|>2$.
  \end{theorem}

  \begin{theorem}\label{th: ChNg II}\cite[Theorem II]{ChNg} For $H$
  a nonsemisimple Hopf algebra of dimension $4p$ where $p \leq 11$
  is an odd prime, then $H$ or $H^\ast$ is pointed.
  \end{theorem}

 \subsubsection{Applications of counting arguments}
 We end this section by applying the preceding preliminary material
 to give some simple alternate arguments for  known facts about
 nonsemisimple nonpointed Hopf
 algebras of dimension $4p$.

\begin{prop}\label{prop:dim4p-groupp} Let $H$ be a nonsemisimple
nonpointed Hopf algebra of dimension $4p$. Then
 $|\GH|  \notin \{ p,2p\}$.
\end{prop}

\begin{proof}
Suppose that  $|\GH| = p$ and then $ \Ho = \Bbbk\GH \oplus (\oplus_{i=1}^{t}
D_i^{n_{i}} ) $
where the $D_i$ are simple subcoalgebras of dimension $d_i^2$ with
 $1 < d_1 <d_2 < \ldots <d_t$. If $p$ divides some $d_i$, then $\dim
H_0 \geq p + p^2 = p(1+p) \geq 4p$ since $p\geq 3$, a contradiction.
Thus $(p,d_i)=1$ for all $i$ and $p$ must divide $n_i$ by Lemma
\ref{lema:andrunatale}. Then $\dim H_0 \geq p + 4p >4p$, a
contradiction. An analogous proof shows that $|G(H) | \neq 2p$.
\end{proof}

\begin{prop}\label{dim4p-coalg-stable}
Suppose that $H$ is generated as an algebra
by a simple coalgebra $D$ of dimension $4$ which is stable
by the antipode. Then $H^{*}$ is pointed.
\end{prop}

\pf
By Proposition \ref{prop:natale-stefan}, $H$ fits into a central
exact sequence of Hopf algebras
\begin{displaymath} \Bbbk^{G}
\hookrightarrow H \twoheadrightarrow A,\end{displaymath}
with $A$ nonsemisimple and $A^{*}$  pointed.
Since $\dim A$ divides $4p$, then $\dim A \in \{ 1,2,4,p,2p,4p  \}$.
We will show that
  $\dim A = 4p$ so that $H \cong A$ and thus
$H^\ast$ is pointed.
\par Since $H$ is nonsemisimple, $\dim A >1$.
If $\dim A = 2, p$ or $2p$, then
$A$ is semisimple by the classification of Hopf algebras of
these dimensions, see \cite{stefan}, \cite{W}, \cite{Z}, \cite{Ng3}.
Since $\Bbbk^{G}$ is  semisimple, this would imply that
$H$ is also semisimple, a contradiction.

\par If $\dim A = 4$, then $\dim \Bbbk^{G} = p$
and this implies
that $ \Bbbk^{G} \simeq \Bbbk C_{p}$.
Thus $C_{p} \subseteq \GH$
so that $|\GH| = p$ or $2p$, a contradiction by
  Proposition \ref{prop:dim4p-groupp}.
Thus $A = H$ and   $H^\ast $ is pointed.
  \epf

\section{Hopf algebras of dimension $rpq$} In this section, $H$
will be a Hopf algebra of dimension $rpq$, with $r<p<q$   primes.
 Recently, Etingof, Nikshych and Ostrik \cite{eno-08},
finished the classification of the semisimple Hopf algebras of
dimension $rpq$ and $rp^{2}$. Specifically, they prove that all
semisimple Hopf algebras of these dimensions can be obtained as
abelian extensions (Kac Algebras). Then, the classification follows by a
result of Natale \cite{pqq}.
 \textbf{Thus, we
will assume that $H$ is nonsemisimple.} One purpose of this section
is to apply counting arguments in the style of D. Fukuda as we did
in \cite{bg}.

\begin{remark}\label{rm: rpq not gen by C}
Recall that by Lemma \ref{lema: dim H mn rel pr} a nonsemisimple
Hopf algebra $H$ of dimension $rpq$ is   nonpointed, has no pointed
sub-Hopf algebras and has no pointed
  quotient Hopf algebras. In particular,  $H$ cannot be generated by  a simple
$4$-dimensional subcoalgebra $C$ stable under the antipode. For then
by Proposition \ref{prop:natale-stefan}, $H \cong \Bbbk^G$, which is
semisimple, a contradiction.  \qed
\end{remark}

\par Also $H$ cannot have the Chevalley property.
  The proof is based
on the proof of \cite[Lemma A.2]{andrunatale}.

\begin{prop}\label{prop:no-chev-rpq}
No nonsemisimple Hopf algebra $H$  of dimension $rpq$ has the
Chevalley property.
\end{prop}

\pf Suppose that $H$ has the Chevalley property. Then $\dim \Ho
\vert \dim H$ and since $H$ is not pointed or cosemisimple, $1 <
\dim \Ho < \dim H$.   Then $\dim \Ho = st$, where $s,t \in \{r, p,q
\}$ and $s< t$.  But by \cite{EG}, \cite{So}, \cite{pqq} or if $s= 2$ by \cite{Ng3},
  all semisimple Hopf algebras of dimension
$st$ are trivial, \textit{i.e.} isomorphic to a group algebra or the
dual of a group algebra.   Hence $\Ho \simeq \Bbbk^{F}$ with $F$ a
non-abelian group of order $st$; in particular, $s\vert (t-1)$.
Consider now the coradical filtration on $H$ and the associated
graded Hopf algebra $\gr H$. Then write $\gr H \simeq R\# \Bbbk^{F}$
with $R$ the \textit{diagram} of $H$. Then $(\gr H)^{*} \simeq
R^{*}\# \Bbbk F$, which implies that $(\gr H)^{*}$ is pointed. This
cannot occur, since $\dim (\gr H)^{*} = \dim \gr H = \dim H = rpq$.
Hence, $\Ho$ cannot be a sub-Hopf algebra. \epf

We note that all Hopf algebras of dimension
$30 = 2 \cdot 3 \cdot 5$ are group algebras
 or duals of group algebras by
\cite{fukuda-30} but the classification of the other Hopf algebras
of dimension $rpq$ with $rpq <100$, namely dimensions $42,66,70$ and
$78$ is still open. We make a few observations about these cases.

\begin{obs}\label{rmk:dimrpq-properties} (i) A nonsemisimple Hopf
algebra of dimension $2pq$ cannot have a semisimple sub-Hopf algebra
$A$ of dimension $pq$.  For if this were the case, there would be a
Hopf algebra epimorphism $\pi: H^\ast \rightarrow A^\ast$ and we
apply Proposition \ref{pr: HNg 1.3}. Since $H^\ast$ has no
nontrivial skew-primitive elements, $(H^\ast)^{co\pi} = \Bbbk C_2$
and we have an exact sequence of Hopf algebras $ \Bbbk C_2
\hookrightarrow H^\ast \twoheadrightarrow A^\ast$.  Thus if $A$, and
thus $A^\ast$, were semisimple,   $H^\ast$ and $H$ would be also.

\par (ii) Suppose that $H$ is nonsemisimple of dimension $2pq$
where all Hopf algebras
of dimension $pq$ are  semisimple. Then by (i) above, $H$ has no
sub-Hopf algebras of dimension $pq$.  Suppose that $|G(H)|
>2$ and let $C$ be a simple subcoalgebra of dimension greater
than $1$. We will show that $C$ generates $H$.

Indeed, let $\mathcal{C} := \langle C \rangle $ be the sub-Hopf algebra generated by
$C$.  Then $\dim \mathcal{C} \in \{2p,2q, 2pq   \}$. If $\mathcal{C}
\neq H$, then $\mathcal{C} \cong \Bbbk^{\mathbb{D}_m}$ with $m \in \{p,q\}$.
Then $H$ is generated by $\mathcal{C}$ and $\Bbbk G(H)$, so  by
Remark \ref{rm: LR2}, $H$ is semisimple, a contradiction.
\end{obs}

\begin{lema} Suppose that $H$ has dimension $2pq$ with $2 <p < q$. Then
\begin{itemize}
\item[(i)] $|G(H)| \neq pq$.
\item[(ii)] If $p \leq 7$ then $|G(H)| \neq 2q$ and  if $q \leq 7$ then $|G(H)| \neq 2p$.
\item[(iii)] If $p \leq 5$ then $|G(H)| \neq q$.
\end{itemize}
\end{lema}

\begin{proof}(i) The statement was proved in  Remark
\ref{rmk:dimrpq-properties}(i).
\par (ii) If $|G(H)| = 2p$, since for all $d$, by Lemma \ref{lema:andrunatale}, $2p$
divides $\dim H_{0,d}= nd^2$ for some $n \geq 1$ then
  $\dim
H_0 \geq 2p + 4p = 6p$. Then by Proposition \ref{prop:biti-dasca}(ii) and Lemma \ref{lema:andrunatale},
$\dim H \geq
6p + 2p(5)  + 4p  = 20p$, a contradiction if $q \leq 7$.

If $|G| = 2q$ and $p \leq 7$ the argument is the same.

\par (iii) Assume $|G(H)| =q$.
For $G, \D$ as above, $\dim H_0 \geq q + 4q = 5q$,
$2\dim P_1^{G,\D}\geq 4q$, $\dim P^{G,G}$ and
$\dim P^{\D,\D}$ must be divisible
by $q$ and so $\dim H \geq 11q > 2(5)q$, a contradiction.
\end{proof}

\begin{cor}\label{cor: not order of gplikes}

\begin{enumerate}
 \item[(i)] If $ \dim H =42$,
  then $|G(H) | \notin \{21, 14,7,6 \}$.
\item[(ii)] If $\dim H = 70$, then $|G(H)| \notin \{ 35,14,10,7  \}$.
\item[(iii)] If $\dim H
= 66$, then $|G(H) | \notin \{ 33, 22,11\}$.
\item[(iv)] If $\dim H = 78$, then $|G(H) |
\notin \{ 39,26, 13  \}$. \qed
\end{enumerate}
\end{cor}

Next we show that for Hopf algebras of dimension $66$, $G(H)$ does
not have order $6$.

\begin{lema} If $\dim H = 6p$ with $p < 13 $, then $|G(H)| \neq
6$.
\end{lema}

\begin{proof} First we suppose that $H$ has a simple subcoalgebra of
dimension $4$ and consider various cases. Let $G:=G(H)$, the group
of grouplikes of $H$ of order $6$,  and let $\D$ denote the set of
simple subcoalgebras of dimension $4$.

\par (i) Suppose that $H_0 = \Bbbk G \oplus \M^\ast(2,\Bbbk)^3$ so
that $\dim H_0 = 18$.  Since, by Remarks \ref{rm: rpq not gen by C}
 and \ref{rmk:dimrpq-properties},
no $D \in \D$ is stable by the antipode, then no $D \in \D$ can be
fixed by $S^2$ either. Thus if $P_1^{1,D}$ is nondegenerate, so are
$P_1^{1,S^2(D)}$, $P_1^{S(D),1}$ and $P_1^{S^3(D)=D,1}$ and $ 2 \dim
P_1^{G,\D} \geq 2(6)4 = 48$. Since $P^{G,G}$ has nonzero dimension
divisible by $6$ and $P^{\D,\D}$ has nonzero dimension divisible by
$12$, then the dimension of $H$ is at least $18 + 48 + 6 + 12 = 84$,
a contradiction.

\par (ii) Suppose that $H_0 = \Bbbk G \oplus \M^\ast(2,\Bbbk)^{3t}$ with
$t \geq 2$ so that $\dim H_0 = 6 + 4(3t) \geq 30$. Since for some integers
$\ell,m, n \geq 1  $, $2 \dim
P^{G,\D} =   24 \ell$, $\dim P^{G,G} = 6m$, $\dim P^{\D,\D} =
12n$, then $\dim H \geq 72$, so that we obtain a
contradiction if $p <13$.

\par (iii) Suppose that
$H_0 = \Bbbk G \oplus \M^\ast(2,\Bbbk)^{3t}
\oplus E_1 \ldots \oplus E_N $  where $t,N \geq 1$ and
 the $E_i$ are
simple subcoalgebras of dimension greater than $4$. Let
$\mathcal{E}$ denote the set of $E_i$ and $\mathcal{D}$ the set of
simple subcoalgebras of dimension $4$.  Then $\dim H_0 \geq 6 + 12 +
18 = 36$. If $P^{G,\mathcal{E}} \neq 0$, then $2 \dim
P^{G,\mathcal{E}} \geq 2(6)(3) = 36$, $\dim
P^{\mathcal{E},\mathcal{E}} \geq 9$ and so $\dim H \geq 81$,
contradiction.  Thus $P^{G,\D} \neq 0$.  If $t=1$, then as in (i)
above $2 \dim P_1^{G,\D} \geq 2(6)4 = 48$, so that $\dim H \geq 36 +
48 = 84$, a contradiction.  If $t \geq 2$, then $\dim H_0 \geq 48$
and
 $2 \dim P_1^{G,\D} \geq 24$ so that $\dim H_1 \geq 72$.
 But $P_2^{\D,\D}$, $P_2^{G,G} $ are
 nondegenerate, so that $\dim H_2 \geq 80$,
 a contradiction.

 Now suppose that $H$ has no simple subcoalgebras of
 dimension $4$ and $H_0 = \Bbbk G \oplus E_1 \ldots \oplus E_t$ where the
 $E_i$ are simple subcoalgebras of dimension at least $9$
 so that $\dim H_0 \geq 6 + 18 =
 24$. Let $\mathcal{E}$ denote the set of simple coalgebras $E_i$.
Then $2 \dim P_1^{G,\mathcal{E}} \geq 2(6)(3) = 36$, $\dim P^{G,G}
\geq 6$, $\dim P^{\mathcal{E},\mathcal{E}} \geq 9$ and must be
divisible by $6$ so that $\dim P^{\mathcal{E},\mathcal{E}} \geq 12$.
But also $\dim P^{\mathcal{E},\mathcal{E}} $ must be a sum of
squares larger than $4$ so that $\dim P^{\mathcal{E},\mathcal{E}}
 >12$. Thus $\dim H > 24 + 36 + 6 + 12 = 78$, a contradiction.
 \end{proof}

 Note that in the proof above, the only place where $p \neq 13$ was
 used was in case (ii). There if $p=13$ we must have that $\ell=n=1$
 and $m=2$.

 \par Next we  show that for dimension $70$ the group of grouplikes must
 have order $1$ or $2$.

 \begin{lema} \label{lem:dimH70Gn5}
If $\dim H = 70$   then $G(H) \ncong C_5$.
 \end{lema}

 \begin{proof} Again, we suppose first that $H$ has a simple
 subcoalgebra of dimension $4$ and consider various cases.  Let $\D$
 denote the set of simple subcoalgebras of dimension $4$ and let
 $G:= G(H) \cong C_5$.
 \par (i) Suppose that $H_0 = \Bbbk C_5 \oplus D_1 \oplus \ldots \oplus
 D_5$ where $D_i \cong \M^\ast(2, \Bbbk)$.
Since no $D_i$ is stable under the antipode,
 we may assume that $S(D_i) = D_{i+1}$,
subscripts modulo $5$.  For if $S^2(D_1) = D_1$, then $S$ would
 permute $D_3,D_4,D_5$. But by Corollary \ref{cor: not order of
 gplikes}, $|G(H^\ast)| \in \{ 1,2,5 \}$ and so $3$ does not divide
 the order of $S$.  Thus  by Proposition \ref{prop:biti-dasca}(i),
 $P_1^{1,D_i}$ is nondegenerate for all $i$
 and $2 \dim P^{G,\D} \geq 2(5)(10) = 100$, a contradiction.

 \par (ii) Suppose that $H_0 =
{\Bbbk C_5 \oplus } \mathcal{M}^\ast(\Bbbk, 2)^{5t}$ where  $t>1$. Then $\dim
 H_0 \geq 5 + 10(4) = 45$, $2 \dim P^{G,\D} \geq 2(5)(2) = 20$,
 $\dim P^{G,G} \geq 5$, and $\dim P^{\D,\D} \geq 4$, so that $\dim H
 \geq 74$.

 \par (iii) Let $H_0 = \Bbbk C_5 \oplus \M^\ast(2, \Bbbk)^{5t} \oplus
 E$, where $t \geq 1$ and $0 \neq E $ is a sum of
simple subcoalgebras $E_i$ of dimension
 greater than $4$. Let $\mathcal{E}$ denote the set of $E_i$. If the
 dimensions of any of the $E_i$ are  relatively prime to $5$, then $\dim
 H_0 \geq 5 + 20 + 5(9) = 70$, a contradiction.
 The only remaining case is $H_0 = \Bbbk C_5 \oplus \M^\ast(2,
 \Bbbk)^{5}\oplus \M^\ast(5,\Bbbk)$; here $t=1$ or else $H=H_0$.
 Then $\dim H_0 = 50$, $2 \dim P_1^{G,\D} \geq 2(5)(2) = 20$ and this
 is a contradiction since $H \neq H_1$.

 \par Thus $H$ cannot have a simple subcoalgebra of dimension $4$.
 The only other possibilities for $H_0$ are $H_0 = \Bbbk C_5 \oplus
 \M^\ast(3, \Bbbk)^5$ and $H_0 =\Bbbk C_5 \oplus \M^\ast(5,
 \Bbbk)^t$ with $t=1,2$. In the first case, $\dim H_0 = 50$, and for
 $\mathcal{E}$ the set of simple subcoalgebras of dimension $9$,
 $2 \dim P^{G, \mathcal{E}} \geq 2(5)(3) = 30$, a contradiction.  In
 the second case, first let $t=1$ and here let $\mathcal{E}$ be the
 set of simple subcoalgebras of dimension $25$. Then $\dim H_0 =30$
 and $2 \dim P^{G,\mathcal{E}} \geq 2(5)(5) = 50$, a
 contradiction.  The proof for $t=2$ is the same.
 \end{proof}

\begin{cor} If $H$ is a nonsemisimple Hopf algebra of dimension
$70$, then $|G(H)|=1,2$. \qed
\end{cor}

\begin{obs}\label{rem: not order of gplikes1}
(i) To summarize, we have that for $H$ of dimension $42,  66 $,
$|G(H)| \in \{ 1,2,3 \}$, for $H$ of dimension $70$, $|G(H)| \in \{
1,2 \}$ and for $H$ of dimension $78$, $|G(H)| \in \{1,2,3,6   \}$.
\par (ii) If $ \dim H=42$   and $G(H) \cong C_3$, then dimension arguments
such as those above show that $H$ has following form: $H_0
\cong \Bbbk C_3 \oplus C$ with $C \cong \mathcal{M}^\ast(3,\Bbbk)$,
$2\dim P^{G,C} = 18$, $\dim P^{G,G} = 3$, $\dim P^{C,C} = 9$.
\end{obs}

\section{Hopf algebras of dimension $8p$}\label{sec:8p}
In this section we prove some results for Hopf algebras of dimension $8p$.
\subsection{Hopf algebras of dimension $8$} \label{dim8}  \label{subsect: 8}
 The structure of Hopf algebras of dimension $8$ or dimension $4p$
naturally plays a role in
 the classification of Hopf algebras of dimension $8p$. Hopf
 algebras of dimension $4p$ were discussed in Section \ref{sect:
 4p}, including a complete description of the pointed ones.

For dimension $8$ the semisimple Hopf algebras are group
algebras, duals of group algebras or the noncommutative noncocommutative
semisimple Hopf algebra of dimension $8$, denoted by
$A_8$ \cite{ma-6-8}. This Hopf algebra is self-dual and $G(A_8)
\cong C_2 \times C_2$; it is constructed as an extension of $\Bbbk
[C_2 \times C_2]$ by $\Bbbk C_2$.  All other Hopf algebras of
dimension $8$ are pointed
   or basic.

\par Let $\xi$ be a primitive $4^{th}$ root of 1.  By \cite{stefan},
  every pointed nonsemisimple Hopf algebra of dimension $8$ is
isomorphic to exactly one of the Hopf algebras listed below:
$$
\begin{array}{ll}
{\mathcal A}_2:=&\Bbbk\langle g,x,y\mid
g^2-1=x^2=y^2=gx+xg=gy+yg=xy+yx=0\rangle,
\\ \noalign{\smallskip}
&\com(g)=g\ot g,\quad\com(x)=x\ot 1+g\ot x,\quad\com(y)=y\ot 1+g\ot y.
\end{array}
$$
$$
\begin{array}{l}
{\mathcal A}'_4:=\Bbbk\langle g,x\mid g^4-1=x^2=gx+xg=0\rangle,
\\ \noalign{\smallskip}
\hspace{2cm}\com(g)=g\ot g,\quad\com(x)=x\ot 1+g\ot x;
\\ \noalign{\vspace{.3cm}}\end{array}
$$
$$
\begin{array}{l}
 {\mathcal A}''_4:=\Bbbk\langle g,x\mid g^4-1=x^2-g^2+1=gx+xg=0\rangle,
\\ \noalign{\smallskip}
\hspace{2cm}\com(g)=g\ot g,\quad\com(x)=x\ot 1+g\ot x;
\\ \noalign{\vspace{.3cm}}\end{array}
$$
$$
\begin{array}{l}
{\mathcal A}'''_{4,\xi}:=\Bbbk\langle g,x\mid g^4-1=x^2=gx-\xi xg=0\rangle,
\\ \noalign{\smallskip}
\hspace{2cm}\com(g)=g\ot g,\quad\com(x)=x\ot 1 + g^2\ot x;
\\ \noalign{\vspace{.3cm}}\end{array}
$$
$$
\begin{array}{l}
{\mathcal A}_{2,2}:=\Bbbk\langle g,h,x\mid g^2=h^2=1, \,
x^2=gx+xg=hx+xh=gh-hg=0\rangle,
\\ \noalign{\smallskip}
\hspace{2cm}\com(g)=g\ot g,\quad\com(h)=h\ot h,\quad\com(x)=x\ot 1 + g\ot
x.
\end{array}
$$

Except for $ {\mathcal A}''_4$, these pointed Hopf algebras
have pointed duals. We have
the following isomorphisms:
${\mathcal
A}_2\simeq({\mathcal A}_2)^*$, ${\mathcal
A}'''_{4,\xi}\simeq{\mathcal A}'''_{4,-\xi}\simeq({\mathcal A}'_4)^*$
 and ${\mathcal A}_{2,2}\simeq({\mathcal A}_{2,2})^*$ \cite{stefan}.
  Moreover, one can check case by case
that ${\mathcal A}_2 $, ${\mathcal A}'''_{4,\xi} $
 and ${\mathcal A}_{2,2} $    have sub-Hopf algebras isomorphic to
$H_{4}$ and ${\mathcal A}'_4, {\mathcal A}''_4$ do not.

\par Let $\mathcal{K} = (\matha''_{4})^{*}$.
Up to isomorphism $ \mathcal{K} $
is the only Hopf algebra of dimension $8$   which
is neither semisimple nor pointed. The next remark is essentially \cite[Lemma 3.3]{GV}.

\begin{obs}\label{lemma:desc-A}
(i)  $\mathcal{K}$ is generated as an algebra by the elements $a,b,c,d$
satisfying the relations
\begin{align*}
ab & = \xi ba & ac & = \xi ca & 0& =cb=bc &  cd & = \xi dc&
bd & = \xi db\\
ad & = da &  ad &=1
& 0& =b^{2}=c^{2} & a^{2}c & = b &  a^{4} & =1
\end{align*}
 \par (ii)  The elements $a=e_{11},b=e_{12},c=e_{21},d=e_{22}$
form a matrix-like basis and
$$
 \com(a^{2}) = a^{2}\ot a^{2}\text{ and }
\com(ac) = ac\ot a^{2} + 1\ot ac.
$$
\par (iii)  $\mathcal{K} \simeq H_{4}\oplus \mathcal M^{\ast}(2,\Bbbk)$
as coalgebras.
\end{obs}

Using Remark \ref{lemma:desc-A}, one sees
that $\mathcal{K}$ is a finite dimensional quotient of the quantum
group $\Oc_{\xi}(SL_{2})$; this is consistent with
  Proposition \ref{prop:natale-stefan}.

\subsection{Nonsemisimple Hopf algebras of
dimension $ 8p $}

Throughout this section, unless otherwise stated,  we will assume
that $H$ is a nonsemisimple nonpointed nonbasic Hopf algebra of dimension $8p$.
 Also recall that $p$ denotes an odd prime.  Our strategy will
be to study the possible orders for the grouplikes in $H$ where
$\dim H = 8p$.  In this  section we prove Theorem \ref{thm:8p}.

\subsubsection{Group of grouplikes divisible by $p$}
In this subsection we concentrate on general results for
  Hopf algebras of dimension $8p$ with $|G(H)|$ divisible by
$p$.

\begin{prop}
\label{pr: dim not p,8p,4p}
$|G(H)| \neq 8p, 4p $ or $p$.
\end{prop}

\begin{proof}
For $H$ non-cosemisimple, $|G(H)| \neq 8p$.   If $|G(H) | = 4p$,
since $H$ is not pointed, $H_0 = \Bbbk G(H) \oplus E$  with
$E$ the sum of simple coalgebras of dimension bigger than 1. Since
$4p$ must divide $\dim(E)$ by Lemma \ref{lema:andrunatale},
then  we must have $H=H_0$, impossible because
  $H$ is not cosemisimple.

\par If $|G(H)|=p$,  then $H$ has no nontrivial skew-primitives
by Lemma \ref{lema: dim H mn rel pr}. Now
we use counting arguments as in the previous sections. Suppose that $H_0 =
\Bbbk G(H) \oplus D_1^{s_1} \oplus \ldots D_t^{s_t} $ for $D_i$   simple of
dimension $n_i^2$ and $ 2 \leq n_1 < \ldots <   n_{t}$.  Let $\D$ denote the set
of simple coalgebras $D_i$. Then by  Proposition
\ref{prop:biti-dasca}(i) and Lemma \ref{lema:fukuda}(ii), $2 \dim P^{C_p, \D}
\geq 2pn_1 $. If $p$ divides $n_1$, then $\dim H \geq \dim H_0 + 2\dim P^{C_p,\D}
\geq p + p^2 + 2p^2 = p(1 + 3p) >8p$ since $p \geq 3$.
 If $(p,n_1) = 1$, then $p $ divides $s_1$ and $\dim H \geq p + 4p + 4p >8p$.
  Thus  in each case, we arrive at a contradiction.
\end{proof}

Thus, if $|G(H)| = 2p$, then $H$ cannot have the Chevalley property.
 For,
suppose that $H_0$ is a sub-Hopf algebra of $H$.  Since $H$ is not pointed or cosemisimple,
 $\dim H_0 = 4p$. Since the semisimple Hopf algebras $\mathcal{G}_i$ have grouplikes of order $4$,
then $H_0 \cong \Bbbk^\Gamma$ where $\Gamma$ is a nonabelian group of order $4p$.  But then $H_0 =
\Bbbk^ \Gamma \cong \Bbbk G(H_0) \oplus D$ where $|G(H_0)| = 2p$ and $D$ is a sum of simple coalgebras.
By Lemma \ref{lema:andrunatale}, $D$ is a sum of matrix coalgebras of dimension $d>1$ and $2p = nd^2$
which is impossible.

\begin{remark}\label{rm: 3 5 7} As in the proof above,
we  use Lemma \ref{lema: dim H mn rel pr}
 together with counting arguments to eliminate
the possibility that $|G(H)| =8$ for some small dimensions.
 Let $\dim H = 8p$ with $p \in \{3,5,7 \}$ and suppose $|G(H)| = 8$.
  By Lemma \ref{lema: dim H mn rel pr},
$H$ has no nontrivial skew
primitive elements. Since  $\dim H_0 = 8 + 8m$ for some
integer $m \geq 1$, by Lemma \ref{cor:bitidasca-p1}(ii), we have that
$\dim H \geq 16 + 40 + 4 = 60$.
\end{remark}

The next proposition shows that  type $(8, 2p)$ is impossible.
\begin{prop}\label{prop:group-8-p}
  If $\vert G(H)\vert=2p$ then  $  H^{*} $ has no
semisimple sub-Hopf algebra $L$ of dimension $8$.
\end{prop}

\begin{proof} Suppose $H^{*}$ contains a semisimple sub-Hopf algebra $L$ of dimension $8$ and
let $\Gamma$ be a subgroup of $ G(H) $ of order $p$. Since $L^\ast$
is semisimple   and  has no grouplike elements of order $p$,  $ \Bbbk
\Gamma \hookrightarrow H \twoheadrightarrow L^\ast$ is an exact sequence
of Hopf algebras. This implies
that $H$ is semisimple, a contradiction.
\end{proof}

 The next proposition determines the coalgebra structure of $H$ when
 $|G(H)| = 2p$.

\begin{prop}\label{pr: grouplikes 2p}
Suppose $|G(H)| = 2p$.
\begin{enumerate}
 \item[(i)]
$H$ contains a pointed
sub-Hopf algebra $\matha$ of dimension
$4p$ and as a coalgebra $H \cong \matha \oplus \mathcal{M}^\ast
(2,\Bbbk)^p$.
 \item[(ii)] If $H^\ast$ is generated by a simple subcoalgebra
of dimension $4$ fixed by $S_{H^\ast}^4$ then
$S_H$ has order $4$.
\end{enumerate}
\end{prop}

\begin{proof}(i) Since $H$ is not pointed,
$H_0 = \Bbbk G(H) \oplus D_1 \oplus
\ldots \oplus D_t$ where the $D_i$ are
simple coalgebras of dimension greater than $1$.
Suppose that $H_{0, mp} \neq 0$ where $m \geq 1$.
Then ${\rm dim}(H_{0,mp}) \geq 2p^2$ and thus ${\rm dim}(H_0) \geq
2p + 2p^2 = p(2 + 2p) \geq 8p$ since $p \geq 3$, and this is
impossible since $H$ is nonsemisimple. If $H_{o,d} \neq 0$ for
$(d,p) = 1$ and $d
>2$ then ${\rm dim}(H_0) \geq 2p + pd^2 = p(2+d^2) > 8p$ which is
also impossible. Thus $D_i =
 \mathcal{M}^\ast(2, \Bbbk)$ for all $i$, and
$H_0 \cong \Bbbk G(H) \oplus    \mathcal{M}^\ast(2, \Bbbk)^p$ as
    coalgebras.

By Proposition \ref{prop:biti-dasca}(ii), $H$ has a nontrivial skew-primitive
$x$ and $x$ together
with $G(H)$ generates a pointed sub-Hopf algebra $\matha$ of $H$ of
dimension $4p$ and (i) is proved.

(ii) By (i) there is a Hopf algebra projection
$\pi: H^\ast \rightarrow \matha^\ast$ for $\mathcal{A}$
the pointed Hopf algebra of dimension
$4p$ from (i).
Then $S_{\matha}$ and $S_{\matha^\ast}$ have order $4$.
Suppose $D \cong \mathcal{M}^\ast(2,\Bbbk)  \subset H^\ast$
is stable under $S_{H^\ast}^4$ and generates $H^\ast$, and suppose
that $S_{H^\ast}^4$ has order $N>1$.  Let $\mathbf{e}$
be a multiplicative matrix for $D$ as in
Theorem \ref{thm:stefan} such that $S_{H^\ast}^4(e_{ij})
= \omega^{i-j}e_{ij}$ where $\omega$ is a primitive $N^{th}$
root of unity.  Then if $i \neq j$, $\pi(e_{ij}) = 0$ and thus
$\dim \pi(D) <3$.  By Theorem \ref{thm:2x2-matrix},
$\pi(D) \subseteq G(\matha^\ast)$ so that $\pi(D)$
does not generate $\matha^\ast$, contradicting
the fact that $D$ generates $H^\ast$.
\end{proof}

Next we show that if $|G(H)| = 2p$, then $H^{*}$ cannot
contain a copy of the Sweedler Hopf algebra.

\begin{prop}\label{prop:2p-sweedler}
Assume $ |G(H)| = 2p$.
Then $ H^{\ast}$ has no sub-Hopf algebra isomorphic to $H_4$.
\end{prop}

\begin{proof}
If $H^*$ contains a sub-Hopf algebra
isomorphic to $H_{4}$, there exists   a Hopf algebra
epimorphism $\pi: H \to H_{4}$. Then,
by Lemma \ref{lm: dim H co dim B = dim H},
 $\dim H^{\co \pi} = \dim\ ^{\co \pi}H = 2p$.
Let $G(H) = \langle c \rangle \cong C_{2p}$
 and let $\Gamma = \langle c^2 \rangle \cong C_p$.
  Since $p$ is odd, we have that
$\Bbbk \Gamma$ is included both in $H^{\co\pi}$ and $^{\co\pi}H $.

On the other hand,
Proposition \ref{pr: grouplikes 2p} implies that $H \simeq
\matha \oplus D$ where $D =  D_{1}
\oplus \cdots \oplus D_{p}$, with $D_{j} \simeq
\mathcal{M}^\ast(2,\Bbbk)$,  for
all $1\leq j\leq p$. We will prove that for every $j$, $1 \leq j \leq p$,
  $\dim D_{j}^{\co \pi}\geq 2$ or $\dim\ ^{\co\pi}D_{j} \geq 2$.
  This fact leads to a contradiction.  Indeed, suppose that for $n$ of the $D_j$,
  $\dim D_{j}^{\co \pi}\geq 2$ and for the remaining $p-n$ coalgebras
$D_j$, $\dim\ ^{\co\pi}D_{j} \geq 2$.
  Since $p$ is odd, either $2n >p$ or $2(p-n) >p$ so
that either $\dim ^{co\pi}D >p$ or $\dim D^{co\pi}>p$.
  Since $\Bbbk \langle c^2 \rangle $ lies in both the left
and right coinvariants, this implies
   that either $\dim ^{co\pi}H >2p$ or $\dim H^{co\pi} > 2p$,
and this gives the desired
  contradiction.

Fix a simple subcoalgebra $D_{j}$ and let $K=\langle D_{j}\rangle$,
the sub-Hopf algebra of $H$ generated by $D_{j}$.
Clearly, $\dim K = 8, 2p, 4p$ or $8p$.
We write $\pi$ also for $\pi|_K$ when the meaning is clear.
If $\pi$ maps $K$ onto $H_4$, then the result
follows from Lemma \ref{lema:proj-coalg}; in particular, $ \dim K \neq 8p $.
If $\pi(K) = \Bbbk$, then $\pi|_K = \varepsilon_K$.
Hence for ${\bf d} = (d_{ij})$ a multiplicative matrix for $D_j$,
$\pi(d_{ij}) = \delta_{ij}$ and $D_j$ lies in both the
left and right coinvariants. It remains to consider
the case when $\pi(K) = \Bbbk G(H_4)= \Bbbk \langle g \rangle$
where $g$ generates   $G(H_4) \simeq C_{2} $.

Assume $\dim K = 8$. Since $K$ is nonpointed,  by
Subsection \ref{subsect: 8} we have that
$K = \mathcal{K} = (\mathcal{A}^{\prime \prime}_4)^\ast \cong L \oplus D_j$
as coalgebras
where $L \cong H_4$. Since $G(L) \subset G(H)$, we have that $c^p \in K$.
Suppose $\pi(c^p) =1$. Since $\pi(c)$ is a grouplike element and $|G(H_{4})|=2$,
we have that $\pi(c) = 1$. If $x$ is a nontrivial
skew-primitive in $H_4 \subset K$ such that $\pi(x) = 0$,
then $c,x$ lie in both $^{co\pi}H$ and $H^{co\pi}$,
contradicting the fact that the dimension of the coinvariants is $2p$.
Thus $\pi(c^p) = \pi(c) = g$, $^{co\pi}L = \{1, gx \}$, $L^{co\pi} = \{ 1, x \}$.
Since $\dim {^{co\pi}K} = \dim K^{co\pi} = 4$,
then we must have that $\dim {^{co\pi}D_j} = \dim D_j^{co\pi} = 2$.

Next we will show that if $\dim \pi(K) =2$ then $K$ cannot have
dimension $4p$ or $2p$. Suppose that $\dim K = 4p$.  Then $\dim
K^{co\pi} = 2p = \dim {^{co\pi}K}$ so that $K^{co\pi} = H^{co\pi}$
and the same for the left coinvariants. Thus $\Bbbk \langle
c^2\rangle \cong \Bbbk C_{ p} \subset K$. If $K$ is nonpointed
semisimple, by the classification of semisimple Hopf algebras of
dimension $4p$ in Section \ref{sec: 4p general}, $p$ does not divide
the order of $G(K)$ either if $K$ is the dual of a group algebra or
if $K$ is one of the semisimple Hopf algebras in \cite{G}. If $K$ is
not semisimple then by Theorem \ref{th: ChNg I}, $K$ is pointed, a
contradiction.

Finally, suppose now that $\dim K = 2p$ so that $K \cong
\Bbbk^{\mathbb{D}_{p}}$ and $\dim K^{co\pi} = p$. Let $\tilde{K}  =
\langle K, \Gamma\rangle$ be the sub-Hopf algebra of $H$ generated
by $K$ and $\Bbbk \langle c^2 \rangle$. Since $\tilde{K}$ is
semisimple, then $\tilde{K} \neq H$ and so has dimension $4p$. But
$\tilde{K}$ is then a  nonpointed semisimple sub-Hopf algebra of $H$
of dimension $4p$ with a grouplike of order $p$,
 and this is impossible by the proof in the paragraph above.
\end{proof}

The next proposition shows that type $(2p,r)$ can occur only if $r=2,4$.

\begin{prop}\label{prop:2p-exact}\label{pr: not 2p 2p}
Suppose  $ |G(H)| =2p $. Then
\begin{enumerate}
 \item[(i)] $ H $ fits into an exact sequence
of Hopf algebras
$ \matha \hookrightarrow H \twoheadrightarrow \Bbbk C_{2}$,
where $ \matha $ is a pointed Hopf algebra of dimension
$ 4p $. In particular, $G(H)$ is cyclic.

 \item[(ii)] If $ \matha^{*} $ is nonpointed, i.e., $\matha \cong \matha(-1,1)$, then
$\dim H^\ast_0 = 8p-4$,
 $G(H^{*})\cong C_4$ and $H^\ast$ has a  sub-Hopf algebra
 isomorphic to $\matha_4^{\prime \prime}$.

 \item[(iii)]If $\matha^\ast$ is pointed then
$\dim H^\ast_0 = 4p$ and $|G(H^\ast)|$ is $2$ or $4$.  If $H^\ast$ has a nontrivial
skew-primitive element, then $H^\ast$ has a  sub-Hopf algebra
isomorphic to $\matha_4^{\prime \prime}$.
\end{enumerate}
\end{prop}

\begin{proof}
(i)  Proposition \ref{pr: grouplikes 2p} implies that
$ H\simeq \mathcal{A} \oplus \mathcal{M}^{\ast}(2,k)^{p}  $,
with $\matha$ a pointed Hopf algebra of dimension $4p$.
Dualizing this inclusion we get a Hopf algebra epimorphism
$\pi: H^{*}\twoheadrightarrow \matha^{*}$
and $\dim H^{*} = 2\dim \matha^{*}$. Thus by Proposition
\ref{prop:R-skew},
$R:= (H^\ast)^{co\pi} = \Bbbk\{1,x\}$ with $x$
 a (possibly trivial) $(1, g)$-primitive
element for some grouplike $g\in G(H^*)$ with  $\ord g = 2n$, $n\geq 1$.
Since $\dim H^{*} = 8p$, we have that $\ord g=2,4$ or $2p$.

 Assume $\ord g =2$. If $x$ is a
nontrivial skew-primitive, then by  Proposition
\ref{prop:R-skew},
$H^\ast$ has a sub-Hopf algebra isomorphic to $H_4$ and this is
impossible by Proposition \ref{prop:2p-sweedler}.
Thus $x \in \Bbbk G(H^{*})$ and $R$ is a Hopf algebra
isomorphic to the group algebra $ \Bbbk C_{2} $. In particular,
$H$ fits into the exact sequence of Hopf algebras
$\matha \hookrightarrow H \twoheadrightarrow \Bbbk C_{2}$.

 We show now that the other cases are not possible. Assume
$\ord g = 4$. Then by Proposition
\ref{prop:R-skew}, $x$ must be nontrivial and $H^{*}$
contains a pointed sub-Hopf algebra $L$ of dimension $8$.
By inspection on the pointed Hopf algebras of dimension 8 in Section \ref{dim8},
we must have that $L \simeq \mathcal{A}'_{4}$ and consequently $L^{*}$ is pointed.
This implies that
$\Bbbk C_{p} \hookrightarrow H \twoheadrightarrow L^{*}$ is an exact
sequence of Hopf algebras and by \cite[Thm. 2.1]{GG} $H$ would be pointed,
a contradiction.

 Finally assume that $\ord g = 2p$. Then $ G(H^{*}) \cong C_{2p}$ and
   by  Proposition \ref{pr: grouplikes 2p}$(i)$, as coalgebras
$H^{*}\simeq \mathcal{B} \oplus \mathcal{M}^{*}(2,\Bbbk)^{p}$
with $\mathcal{B}$ a pointed Hopf algebra of dimension $4p$.
Let $\pi$ be the Hopf algebra map from $H$ onto $\mathcal{B}^{*}$.
By Proposition
\ref{prop:R-skew}, we have that $H^{\co \pi} =\Bbbk\{1,x\} $
with $x$ a skew-primitive element.
If $x$ is trivial, we have that
$\Bbbk C_{2} \hookrightarrow H \twoheadrightarrow \mathcal{B}^{*}$
is an exact sequence of Hopf algebras and consequently $\Bbbk C_{2}$ is normal
in $H$. This implies that $\Bbbk C_{2}$ is also normal in $\mathcal{A}$ and
$\mathcal{A}$ fits into an exact sequence of Hopf algebras
$\Bbbk C_{2} \hookrightarrow \mathcal{A} \to K$, where $\dim K=2p$ so that $K$
is semisimple.  Thus $\mathcal{A}$ is also semisimple, a contradiction.    Therefore
$x \in P_{1}(H)$ but $x \notin H_0$.  Since   $\pi(x) = 0$,
 $\pi(P_{1}(H))=0$ and in this case	  $\mathcal{B}^{*}$
would be the image of the coradical and
 hence   cosemisimple, which is also a contradiction.

%

(ii)
Suppose now that $\mathcal{A}^{*}$ is nonpointed.
Recall from Subsection \ref{sect: 4p} that then
$\mathcal{A} \cong \mathcal{A}(-1,1)$ and
$ \mathcal{A}^{*}\simeq H_{4} \oplus
\mathcal{M}^{\ast}(2,\Bbbk)^{p-1} $ as
coalgebras. Hence $\dim (\matha^{*})_{0} = 4p-2$ and by Proposition
\ref{prop:exact-dim-cor}, $\dim (H^{*})_{0} = 8p-4$. Thus
  $ H^{*} $ contains a nontrivial skew-primitive
element, since otherwise   Proposition
\ref{prop:biti-dasca} gives a contradiction.
Thus $|G(H^\ast)| >1$.   Since
$ \dim H^{*}- \dim (H^{*})_{0}=4 $ is divisible
by $|G(H^\ast)|$ we have that $  |G(H^\ast)|$ is $2$ or $ 4$.
 But if $G(H^\ast) \cong C_2$ or if
 $G(H^\ast) \cong C_2 \times C_2$,
 then $H^\ast$ would contain
a sub-Hopf algebra isomorphic to
$H_4$, and this is impossible by  Proposition
\ref{prop:2p-sweedler}.

\par Thus $H^\ast$ has a pointed sub-Hopf algebra $L$  with $G(L)\cong C_4$
  and so $L$ has dimension $8$.
Then there is a Hopf algebra epimorphism $\rho: H \rightarrow L^\ast$ and $H^{co \rho} \cong \Bbbk C_p$ so that
we have an exact sequence of Hopf algebras
$
\Bbbk C_p  \hookrightarrow H  \twoheadrightarrow L^\ast,
$ and dualizing we obtain the exact sequence
$L \hookrightarrow H^\ast \twoheadrightarrow \Bbbk C_p
$. By Proposition \ref{prop:exact-dim-cor},  $ 6p = \dim H_0 = p \dim L^\ast_0 $.
Thus we must have that $\dim L^\ast_0 = 6$
and $L^\ast $ cannot be pointed.  We must have that $L^\ast \cong \mathcal{K}$
 and $L \cong \matha_4^{\prime \prime}$.

(iii) Now suppose that $\matha^\ast$ is pointed so that
$G(\matha^\ast) \cong C_{ 2p}$ by
Subsection \ref{sect: 4p}. Then
again using Proposition \ref{prop:exact-dim-cor} we have that
$\dim H^\ast_0 = 4p$. If $|G(H^\ast)| \neq 2,4$,
since $H^\ast$ has a grouplike of order $2$,
by Proposition \ref{pr: dim not p,8p,4p} and Proposition \ref{prop:group-8-p},
$|G(H^\ast)|$ must be $2p$.
 Then $H^\ast_0 = \Bbbk C_{2p} \oplus E$
 where $\dim E = 2p$ and $E$ is a sum of
 simple subcoalgebras
 of dimension greater than $1$.
 No simple subcoalgebra can have dimension
 divisible by $p$ since $p^2 > 2p$.
But if a simple subcoalgebra has dimension
$d^2$ with $1 < d $ and $(d,p)=1$ then $H^\ast_0$ must contain at least
$p$ such simple coalgebras and $d^2p>2p$,
a contradiction.

If $H^\ast$ has a nontrivial skew-primitive
element then the same argument as in (ii) above shows that $H^\ast$ has
a sub-Hopf algebra isomorphic to $\matha_4^{\prime \prime}$.
\end{proof}

\begin{cor}\label{cor: 3 5 C4} If
$|G(H)| = 2p$ with $p=3$ or $ 5$,
then   $H^\ast$ has a sub-Hopf algebra isomorphic to $\matha_4^{\prime \prime}$.
\end{cor}
\begin{proof}
 It suffices to show that $H^\ast$ has a nontrivial skew-primitive element
and then the statement follows from Proposition \ref{pr: not 2p 2p}.
We may assume that we are in case (iii)   of Proposition \ref{pr: not 2p 2p},
 so that
$\dim H^\ast_0 = 4p$.

Let $p=3$. By Proposition \ref{prop:biti-dasca}, if $H^\ast$ has
no skew-primitive, then for $|G(H^\ast)| \geq 2$,
 $24 = \dim H^\ast \geq 12 + 2(5) + 4 = 26$,
 a contradiction.

\par If $p = 5$, and $ |G(H^\ast)|=  2$, then $\dim H^\ast_0 = 20$
forces   $H^\ast_0 \cong kC_2 \oplus \mathcal{M}^\ast(3,\Bbbk)^2$.
 Then if $H^\ast$ has no nontrivial skew-primitive,
 Proposition \ref{prop:biti-dasca} implies that
$40 \geq 20 + 2(7) +9 = 43$, a contradiction.  If $p = 5$ and
$|G(H^\ast)| = 4$, then      Proposition \ref{prop:biti-dasca}
implies that  $40 \geq 20 + 4(5) + 4 = 44 $, again a contradiction.
\end{proof}

Now we can give the proof of Theorem A.

\bigbreak \noindent {\bf Proof of Theorem A.} Let $ H $ be a
nonsemisimple Hopf algebra of dimension $ 8p $. By Proposition
\ref{pr: dim not p,8p,4p} we have that $ |G(H) | \in \{1,2,4,8,  2p
\}$. If $ |G(H) |=2p$, then by Proposition \ref{pr: not 2p 2p} we
have that $2 \leq |G(H^\ast)| \leq 4$ and the theorem is proved.
 \qed

 \subsubsection{Further results for some specific primes}
In this section, we improve the results of Theorem A for some
specific primes $p$.

\begin{prop}\label{prop:fixed under left}
Suppose $|G(H)|=2p$,  $G(H^\ast) \cong C_4 = \langle g \rangle$ and
  $H^\ast$ contains a simple subcoalgebra $D$ of dimension $4$.
Also assume that $(H^\ast)_0$ is not a sub-Hopf algebra of $H^\ast$,
i.e., $H^\ast$ does not have the Chevalley property.  Then
\begin{enumerate}
\item[(i)] $D$ generates $H^\ast$ as a Hopf algebra;
\item[(ii)] $D$ is not fixed by $L_{g^2}$, $R_{g^2}$, i.e.,
by left or right multiplication by $g^2$. If $g^2 \notin \mathcal{Z}(H^\ast)$
then   $D$ is also not fixed by the adjoint action of $g^2$.
\end{enumerate}
\end{prop}

\begin{proof}
(i) Let $L = \langle D \rangle$
be the sub-Hopf algebra of $H^\ast$ generated by $D $, then $L$ is
a nonpointed Hopf algebra of dimension $8, 2p,4p$ or $8p$.
We will show that each dimension except $8p$
is impossible.

 Suppose the dimension of $L$ is $8$.
 Then, by Section \ref{subsect: 8},
 either $G(L) \cong C_2 \times C_2$, impossible since $G(H^\ast) \cong C_4$
   or else $L$ contains a copy of $H_4$,
   impossible by Proposition \ref{prop:2p-sweedler}.

Suppose the dimension of $L$ is $2p$ so that  $L \cong
\Bbbk^{\mathbb{D}_p}$. Let $\Ll = \langle L, \Bbbk \langle g \rangle
\rangle$ be the semisimple sub-Hopf algebra of $H^{*}$ generated by
$L$ and by $g$, a generator of  $G(H^\ast)$. Then the dimension of
$\Ll$ is divisible by $2p$ and   by $4$ so it must be $4p$.
Suppose that $\Ll$ is  the self-dual semisimple Hopf algebra
of dimension $4p$ from \cite{G} with grouplikes cyclic of order $4$.
  Then, if $G(H) = \langle h \rangle \cong C_{2p}$  so that $\langle h^2 \rangle \cong C_p$
and $\pi$ is the
Hopf algebra projection from $H$ onto $\Ll^* \cong \Ll$, $\pi(h^{2n}) =1$ for $0 \leq n \leq p-1$.
This
contradicts   Lemma \ref{lm: dim H co dim B = dim H} which states
that the dimension of $H^{co\pi}$ is $2$.

Finally suppose that $\Ll
\cong \Bbbk^\Gamma$ for $\Gamma $ a nonabelian group of order $4p$.
 Then there is a Hopf algebra projection $\pi$ from $H$ onto $\Bbbk
\Gamma$ and by Proposition \ref{prop:R-skew}, $H^{ co\pi} = \Bbbk
\{1,x \}$ where $0 \neq x$ is skew-primitive. If $x$ is
trivial, then $H^{ co\pi} \cong
\Bbbk C_2$. But then by  Lemma \ref{lm: dim H co dim B = dim H}, the
sequence
 $ \Bbbk C_2 \overset{\imath}\hookrightarrow
H\overset{\pi}\twoheadrightarrow \Bbbk \Gamma $ is exact so that $H$
is semisimple, a contradiction.  Thus, by the proof of Proposition \ref{prop:R-skew},
 $x$ is a nontrivial $(1,b)$-primitive with $b=h$ or $b=h^{p}$
 and $x^{2}=0$. Let $b=h$.
 Thus, since $\{ h^ix | 0 \leq i \leq 2p-1 \} $ is a set
 of $2p$ linearly independent nontrivial skew-primitive elements of $H$,
 all with square $0$, $H$  cannot have a pointed sub-Hopf algebra
 isomorphic to $\mathcal{A}(-1,1)$. Let $b = h^p$, a grouplike of order $2$.
 Then $H$ contains a sub-Hopf algebra isomorphic to $H_4$ and again cannot
 have a pointed sub-Hopf algebra isomorphic to $\mathcal{A}(-1,1)$.
 Thus $\mathcal{A}$ has pointed dual
  and so we are in Case (iii)
of Proposition \ref{pr: not 2p 2p}. Then $\dim H^\ast_0 =4p$
 and so $\Ll = H^\ast_0$. Since we assumed that
  $H^\ast$ does not have the Chevalley property, this is a contradiction.

Suppose the dimension of $L$ is $4p$. By its construction $L$ is not
pointed.  Also $L$ cannot be  semisimple by the arguments in the
case above where $\dim \Ll = 4p$. If $L$ is basic then $L \cong
\matha(-1,1)^\ast \cong H_{4} \oplus \m^\ast(2,\Bbbk)^{p-1}$, which
is impossible since $H^\ast$ does not contain a copy of $H_{4}$.
Thus both ${L}$ and ${L}^\ast$ are nonsemisimple, nonpointed so that by
Theorem \ref{th: ChNg I},
 $|G(L)| ,|G(L^{*})|\leq 2$.
But since $L$ is a sub-Hopf algebra of $H^{*}$, we have a Hopf
algebra epimorphism $\pi: H\twoheadrightarrow L^{*}$ with $\dim
H^{\co \pi} = 2$. This implies that $\pi(h^{2}) \neq 1$ and
consequently $p\leq|G(L^{*})|\leq 2 $, a contradiction. Thus, this
case is also impossible and we have proved (i), namely that $L=H$.

(ii)  Now suppose that $g^2L = L$; if $L$ is stable under $R_{g^2}$
or $\ad_\ell(g^2)$ with $g^2 \notin \mathcal{Z}(H^\ast)$, the argument is the same.
Let $\mathcal{A} \subset H$ be the $4p$-dimensional
pointed sub-Hopf algebra of $H$
from Proposition \ref{pr: grouplikes 2p};
there is a Hopf algebra epimorphism $\pi: H^\ast \rightarrow \mathcal{A}^\ast$.
 If $\mathcal{A}^\ast$ is pointed, then $G(\mathcal{A}^\ast) \cong C_{2p}$,
otherwise $G(\mathcal{A}^\ast) \cong C_2$. In either case, $\pi(g^2)
= 1$. Then Lemma \ref{truco util} implies that $\pi(H^\ast)
\subseteq \Bbbk G(\mathcal{A}^\ast)$, and this contradiction
finishes the proof.
\end{proof}

  \begin{cor}\label{cor:2p-4-4p}
  Assume $|G(H)| = 2p$  with
  $p = 3,7,11$ and suppose  that  the $4p$ dimensional pointed sub-Hopf
   algebra $\mathcal{A}$ of $H$ from Proposition \ref{pr: grouplikes 2p}
has pointed dual. If $H^\ast$ does not have the Chevalley property,
then $G(H^\ast) \ncong C_4$.
  \end{cor}

 \begin{proof}Suppose that $G(H^\ast) = \langle g \rangle \cong C_4$.
  With the notation of Proposition \ref{prop:2p-exact},
the assumption that $\matha^\ast$ is pointed means
   that we are in Case (iii) so that $\dim H_0^\ast = 4p$.
With   notation as in Proposition \ref{prop:fixed under left},
  it remains to show that for $p=3,7,11$,   then  $H^\ast$ has  a
  simple subcoalgebra   of dimension $4$ stable
  under   $L_{g^2}$ and that will give a contradiction.

 If $p=3$,   $H^\ast_0 \cong \Bbbk C_4
  \oplus \m^\ast(2,\Bbbk)^2$ so since the order of $g$ is
  $4$, the statement is clear.

 If $p=7$ then either $H^\ast_0 \cong \Bbbk C_4 \oplus
  \m^\ast(4,\Bbbk)
  \oplus \m^\ast(2,\Bbbk)^2$ or else $H^\ast_0 \cong \Bbbk C_4
  \oplus \m^\ast(2,\Bbbk)^6  $ and in either case, the statement follows.

 If $p=11$ then $H^\ast_0 \cong \Bbbk C_4 \oplus D$ where $D$ is one of
the following: $\m^\ast(2,\Bbbk)^{10}$ or
  $\m^\ast(3,\Bbbk)^4 \oplus \m^\ast(2,\Bbbk)$ or
$\m^\ast(4,\Bbbk) \oplus E$ where $E$
  has dimension $24$. Then $E \cong \m^\ast(2,\Bbbk)^6$
or $E \cong \m^\ast(2,\Bbbk)^2 \oplus \m^\ast(4,\Bbbk)$.
  In any case, $H^\ast$ has a simple $4$-dimensional
subcoalgebra stable under $L_{g^2}$.
  \end{proof}

  \begin{cor}\label{cor: 24 type 6,4 corad} Let $\dim H = 24$ and suppose $H$ is of type $(6,4)$
and $H^\ast$ does not have
the Chevalley property.
Then $H$ fits into an exact sequence
  $\matha(-1,1) \hookrightarrow H \twoheadrightarrow \Bbbk C_2$,
in other words, we are in Case (ii) of
  Proposition \ref{pr: not 2p 2p}. Then we have that either
$H_0^\ast \cong \Bbbk C_4 \oplus \mathcal{M}^\ast(2,\Bbbk)^4$
  or else $H_0^\ast \cong \Bbbk C_4 \oplus \mathcal{M}^\ast(4,\Bbbk)$.
  \end{cor}

  \begin{proof}
  The statement follows from Corollary \ref{cor: 3 5 C4}
and Corollary \ref{cor:2p-4-4p}.
  \end{proof}


\subsection{Generalizations of results of Cheng and Ng}

In this section we generalize some results of \cite{ChNg} to study
Hopf algebras of dimension $8p$ with group
of grouplikes of order $2^i$.  We assume throughout this section that
$H$ is nonsemisimple, nonpointed, nonbasic
and has dimension $8p$.

The following propositions are similar to
\cite[Prop. 3.2]{ChNg}.

\begin{prop}\label{pr:boson-2-4}\label{pr:4-4-boson}
\begin{enumerate}
\item[(i)] If $H$ contains a pointed sub-Hopf algebra $K$ of dimension $8$,
then $G(H) = G(K)$.
\item[(ii)] Assume $H\simeq R\# K$ where  $K$ is
pointed and basic of dimension $8$,
and $R$ is a braided Hopf algebra of dimension $p$ in $ \ydk $. Then
$G(H) \cong G( H^{*}) $ so that $H$ is of type $(4,4)$ or type
$(2,2)$.
\item[(iii)] Suppose that $|G(H)| = 2^t$ for $t \in \{1,2,3 \}$
and suppose that $H^\ast$ contains
a sub-Hopf algebra $L$ of dimension $8$ so that there is a Hopf algebra
epimorphism $\pi: H \rightarrow L^\ast$. Then $\pi$ is an injective
Hopf algebra map from $\Bbbk G(H)$ to $\Bbbk G(L^\ast)$.
\item[(iv)] Suppose that $H$ contains a pointed
  sub-Hopf algebra $K$ of dimension $8$
with $|G(K)| = 4$ and $H^\ast$ contains
a pointed   sub-Hopf algebra $L$ of dimension $8$.
 Then $K \cong L^\ast$ and $H \cong R \#K$ where $R$ is a braided Hopf algebra in
$^K_K\mathcal{YD}$ of dimension $p$.
\end{enumerate}
\end{prop}

\begin{proof} (i) Suppose $H$ has a grouplike
element $g$ such that $g
\notin G(K)$.  Then $\langle g, K \rangle$, the sub-Hopf algebra of
$H$ generated by $g$ and $K$, is pointed and has dimension greater
than $8$ and divisible by $8$, so must be all of $H$. This is a
contradiction since $H$ is not pointed.

(ii) Assume $H\simeq R\# K$ with $K$ and $K^{*}$ pointed. By (i),
$G(H) = G(K)$. Since $H^{*}\simeq  R^{*}\# K^{*}$, and
  $K^\ast$ is pointed by assumption, then again by (i), $G(H^\ast) =
G(K^\ast)$. By Section \ref{subsect: 8},  since $K$ is pointed and
basic, then $G(K) \cong G(K^\ast)$ and thus $G(H) \cong
G(H^\ast)$.

(iii) Dualizing the inclusion $L \subset H^\ast$,
we get a Hopf algebra
epimorphism $\pi: H \longrightarrow L^\ast$.
Since $\dim L^\ast = 8$, $\dim R = p$ where
$R = H^{co \pi}$ is the algebra of coinvariants.
Suppose that $\pi(g) = 1$ for some $g \in G(H)$
and let $\Gamma = \langle g\rangle$. Then
$\Bbbk \Gamma \subset R$ and $R$ is
a left $(H, \Bbbk \Gamma)$-Hopf module
where the left action of $\Bbbk \Gamma$ on $R$
is left multiplication.  Then by the Nichols-Zoeller
theorem, $R$ is a free $\Bbbk \Gamma$-module
which is impossible unless $\Gamma = \{1 \}$.
Thus $\pi$ is an injective Hopf algebra
map on $\Bbbk G(H)$ as claimed.

\par (iv)
Let $\pi: H \longrightarrow L^\ast$ and $R = H^{co \pi}$
  as in the proof of (iii). Let $x$ be a nontrivial
$(g, 1)$-primitive in $K$. We wish to show that $\pi(x)$ is a nontrivial skew-primitive in $L^\ast$ and then
$\pi$ will be an isomorphism from $K$ to $L^\ast$, proving the statement.

\par By (i), $G(H) = G(K)$ and $G(H^\ast) = G(L)$.  By (iii), since $|G(H)|=4$, $|G(L^\ast)| \geq 4$, and since $L$
is pointed, by the description of the duals of pointed Hopf algebras of dimension $8$ in Section \ref{subsect: 8}, $L^\ast$
must also be pointed.  Again, by Section \ref{subsect: 8}, $G(L) \cong G(L^\ast)$.
Let $G$ denote $G(H) \cong G(K) \cong G(L) \cong G(L^\ast) \cong G(H^\ast)$.

\par By (iii), $\pi(x)$ is $(g,1)$-primitive. Suppose that $\pi(K) \subseteq \Bbbk G \subset L^\ast$.
Then  $\pi(x) = \lambda(g-1)$ with $\lambda \in \Bbbk$.
But this implies that $\pi(x^2) = \lambda^{2}(g^2 -2g +1)$, which is only possible if $\lambda = 0$ since
$x^2 = 0$ or $x^2 = g^2 -1$. Thus $\Bbbk\{ 1,x \} \subset R = H^{co\pi}$. On the other hand, if
$V$ denotes the vector space with basis $\{ hx^i | h \in G(H), h \neq 1, i = 0,1 \}$, then $V \cap R = \{ 0 \}$.
Since $\dim R =p$,
there is some $0 \neq z \in R$ such that $z \notin K$.  Then $  \langle K,z \rangle$, the sub-Hopf algebra
generated by $K$ and $z$, has dimension greater than $8$ and divisible by $8$ so is all of $H$.
By Lemma \ref{lm: on R pi=epsilon}, $\pi(z) \in \Bbbk$.  Thus $\pi(H) \subseteq \Bbbk G$, a contradiction, and so
$\pi(z)$ is a nontrivial skew-primitive in $L^\ast$.
\end{proof}

\begin{cor} \label{cor: 4.18} Suppose  that $H$ is of type $(4,4)$
 and $H,H^\ast$ each
have a nontrivial skew-primitive element.  Then $H \cong R \# K$ where
$K,K^\ast$ are   pointed Hopf algebras
of dimension $8$ and $R$ is a braided Hopf algebra in $_K^K\mathcal{YD}$
of dimension $p$.
\end{cor}

\begin{proof} By Proposition \ref{pr:4-4-boson}(iv),
it remains only to show that $H$, $H^\ast$ have
pointed sub-Hopf algebras of dimension $8$.
Let $K = \langle G(H),x \rangle$, the sub-Hopf algebra of $H$ generated by $G(H)$ and a
nontrivial skew-primitive element. Then $\dim K <8p$ and
is divisible by $4$   so is either $8$ or $4p$.
Since
all pointed Hopf algebras of dimension $4p$ have group of
grouplikes of order $2p$ (see Section \ref{sect: 4p}), $\dim K = 8$. Similarly $H^\ast$ has a pointed sub-Hopf
algebra of dimension $8$.
\end{proof}

The following proposition follows the proof of
\cite[Thm. 3.1]{ChNg}.

\begin{prop}\label{pr:R-semisimple}
Let $K$ be a Hopf algebra and
$R$ be a braided Hopf algebra in $ \ydk $ of odd dimension.
If the order of the antipode in the bosonization $R\#K$ is a power
of $ 2 $, then $R$ and $R^{*}$ are semisimple.
\end{prop}

\pf
Let $H = R\#K$ be the Radford biproduct or bosonization
of $R$ with $K$. As $R$ is stable by $ \cS_{H}^2 $,
 by \cite[Thm. 7.3]{andrussch}
it suffices to prove that $\Tr(\cS_{H}^{2}|_{R}) \neq 0$.  Clearly, the
order of $ \cS_{H}^{2}|_{R} $ divides the order of $ \cS_{H}^{2}
$ and hence is a power of $ 2 $. If $\Tr(\cS_{H}^{2}|_{R}) = 0 $, then
  by \cite[Lemma 1.4]{Ng}
  $\dim R$ is even, a contradiction. Thus $R$ is semisimple.
The same proof holds for $R^{*}$ since $H^{*} \simeq R^{*}\# K^{*}$.
\epf

Recall  that
$H$ nonsemisimple of
 dimension $24$
 with $|G(H)|=4$ has a nontrivial skew-primitive element by
 Proposition
\ref{prop:biti-dasca}. Then  Corollary \ref{cor: 4.18} and Proposition \ref{pr:R-semisimple}
imply the next statement.

\begin{cor}\label{cor:24-4-4-1}
Suppose $\dim H =24$ and $H$ is of type $(4,4)$.
 Then $H\simeq R\# K$ with $K$ and
$ K^{*} $  pointed Hopf algebras of dimension $ 8 $ and
$R$ a semisimple Hopf algebra in $ \ydk $ of dimension $ 3 $.\qed
\end{cor}

The following lemmata generalize results of Cheng and Ng
used to study $H_4$-module algebras, in particular \cite[3.4,3.5]{ChNg}.

\begin{lema}\label{lem:xe=0}
Let $K$ be a   pointed Hopf algebra generated by grouplikes and skew-primitives,
and
let $ A $ be a finite dimensional left $K$-module algebra. If $A$ is
a semisimple algebra and $ e $ is a central idempotent of $ A $ such
that the two-sided ideal $I=Ae$ is stable by the action of $G(K)$,
then $I$ is a $K$-submodule of $ A $ with $g\cdot e = e$ for all $g\in G(K)$ and $x\cdot
e = 0$  for any skew-primitive element $x$.
\end{lema}

\pf
Write $e=e_{1} + \cdots + e_{t}$ as a sum of orthogonal primitive
central idempotents. Since $I$ is stable under the action of $G(K)$, then
$G(K)$   permutes the primitive idempotents $ e_{1},\ldots ,e_{t} $  and
hence $g\cdot e = e$ for all $g\in G(K)$.

Let  $x$
be a $(1,g)$-primitive. Then $x\cdot e = x\cdot e^{2} = (x\cdot e)e
+ (g\cdot e)(x\cdot e) = 2(x\cdot e)e$.  Thus  $x \cdot e \in I$ so that
$ x \cdot e =  (x \cdot e ) e  $ and then $x \cdot e = 2 (x \cdot e)$ implying that $x\cdot e=0$.
Moreover, since $x\cdot (ae) = (x\cdot a)e$ for all $a\in A$, it
follows that $I$ is stable under the action of $x$ and since $K$ is
generated by grouplikes and skew-primitives, $I$ is a $K$-submodule
of $A$. \epf

\begin{lema}\label{lem:x-R=0}
Let $K$ be a   pointed Hopf algebra  with
abelian group of grouplikes. Let $ A $
be a   semisimple braided Hopf algebra in $\ydk$.
If $I$ is a one-dimensional ideal of $A$, then $x\cdot I =0$ for all
skew-primitive elements $x$ of $K$.
\end{lema}

\pf
Let $x$ be a $(1,g)$-primitive element of $ K $ and
denote by $\overline{K}$ the pointed sub-Hopf
algebra of $K$ generated by $x$ and $g$. Note that since
$G(K)$ is abelian, then $g x g^{-1} = \chi(g) x = \omega x$
for some character $\chi$ of $G$ and $1 \neq \omega$ an $N$-th
root of unity with $N= \ord g$.

Since $A$ is semisimple, $I = Ae_{1} = \Bbbk e_{1}$
for some
central primitive idempotent. Thus we need to prove that
$x\cdot e_{1} =0$. If $g\cdot e_{1}=e_{1}$, then the result follows from
Lemma \ref{lem:xe=0} using $ \overline{K} $ instead of $ K $.
Assume $g\cdot e_{1} \neq e_{1}$ and let $e_{1},\ldots ,e_{t}$ be
representatives of the set $\{g^{i}\cdot e_{1}\}_{0\leq i<N}$
of primitive central
idempotents of $ A $; in
particular $t$ divides $N$ and $g\cdot e_{t}=e_{1}$.
Let $e = e_{1}+\cdots +e_{t}$, then $\overline{I} = Ae$ is
a two-sided ideal of $ A $ which is stable under the action of
$\Gamma = \langle g\rangle$. Hence, by Lemma \ref{lem:xe=0}
we have that $x\cdot e = 0$.
Since $\com(x) = x\ot 1 + g\ot x$, we have that
$x\cdot e_{i} = x\cdot e_{i}^2 = \alpha_{i,i} e_{i} + \alpha_{i,i+1} e_{i+1}$
and $x\cdot e_{t} = \alpha_{t,t} e_{t} + \alpha_{t,1} e_{1}$
for some $\alpha_{ij}\in \Bbbk$. Using
that $x\cdot e =0$ we get that
$\alpha_{i-1,i} + \alpha_{i,i} =0 $ and $\alpha_{t,1} + \alpha_{1,1}=0$
for all $2\leq i \leq t$.
On the other hand, using that $gxg^{-1} = \omega x$ we obtain
that $\omega \alpha_{i,i} = \alpha_{i-1,i-1}$ for all $2\leq i \leq t$,
$\omega \alpha_{i,i+1} = \alpha_{i-1,i}$ for all $2\leq i \leq t-1$
and $\omega \alpha_{1,1} = \alpha_{t,t}$,
$\omega \alpha_{1,2} = \alpha_{t,1}$ and
$\omega \alpha_{t,1} = \alpha_{t-1,t}$.
Hence
$\alpha_{1,2} = \omega^{-1}\alpha_{t,1} = -
\omega^{-1}\alpha_{1,1}$ and $x\cdot e_{1} =
\alpha_{1,1}(e_{1} - \omega^{-1}e_{2})$.

\par Denote by $\lambda_{A}$ the right integral of $ A $. Then
by \cite[Thm. 5.8, Rmk. 5.9]{FMS},
see also \cite[Eq. (3.4)]{ChNg},
for $k \in K$, $a \in A$, we have
$\lambda_{A}(k\cdot a) = \eps_{K}(k) \lambda_{A}(a)$.
Then $g\cdot e_{1} =e_{2}$ implies that
$\lambda_{A}(e_{1})= \lambda_{A}(e_{2})$ and consequently
$$0= \eps_{K}(x)\lambda_{A}(e_{1})=
\lambda_{A}(x\cdot e_{1}) =
\alpha_{1,1}(\lambda_{A}(e_{1})-\omega^{-1}\lambda_{A}(e_{2}))
=\alpha_{1,1}(1-\omega^{-1})\lambda_{A}(e_{1}).$$
This implies that $\alpha_{1,1}=0$, since $\omega^{-1}\neq 1$
and $\lambda_{A}(e_{1})\neq 0$ because the kernel
of a right integral does not contain any nontrivial
ideal. Hence $x\cdot e_{1}=0$ and the lemma is proved.
\epf

\begin{prop}\label{pr: boson implies chevalley}
Suppose that $H \cong R \#K$  where $K$ is a  pointed Hopf algebra  of
 dimension $8$, and $R$ is a
Hopf algebra of dimension $p$ in $_K^K\mathcal{YD}$ such that $x \cdot R =0$  for
some $(1,g)$ primitive $x \in K$, $x \cdot R$ being the adjoint action of $x$ on $R$.

\begin{enumerate}
\item[(i)]
 If $  |G(H) | =4$,   suppose furthermore that $K$ is basic and the   condition above holds for $
  R^\ast $ and $ K^\ast$, i.e.,  for $y$  some nontrivial  $(1,h)$-primitive in $K^\ast$,
then $y \cdot R^\ast =0$.  Then,  $H$ and $H^\ast$ have the Chevalley property
and, in the notation of
Section \ref{dim8}, we have
\begin{enumerate}
 \item[(a)] $G(H) \cong C_4$, and $K\cong \mathcal{A}_{4}^{\prime}$ or
 $K\cong \mathcal{A}_{4,\xi}^{\prime \prime\prime}$; or
 \item[(b)] $G(H) \cong C_2 \times C_2$ and
  $K \cong \mathcal{A}_{2,2}$.
\end{enumerate}

\item[(ii)] If  $|G(H)| = 2$ then there is a Hopf algebra epimorphism $\pi: H \rightarrow A$
where $A$ is a Hopf algebra of dimension $4p$ which is   nonsemisimple, nonpointed and nonbasic. Thus if  $p \leq 11$, this
situation cannot occur.
\end{enumerate}
\end{prop}

\begin{proof} We note that by Proposition \ref{pr:4-4-boson}(i), $G(K) = G(H)$.
  \par (i) Let $J$ be the Hopf ideal of $H$ generated by $x$. Since
 $x \cdot R = 0$, then
$xR = -gRx\subseteq Hx$ and so $J = Hx$. As a $\Bbbk$-space,
$J = Span \{r_ihx :\ r_1, \ldots, r_p \mbox{ a basis for }$ $
R, h \in G(H) \}$ and so the dimension
of $J$ is at most $4p$. Then $\dim H/J \geq 4p$ and divides $8p$ since
$(H/J)^\ast$ is isomorphic to
a sub-Hopf algebra of $H^\ast$. Thus $\dim J = \dim H/J = 4p$.
Then there is a Hopf algebra
epimorphism $\pi: H \rightarrow A$ where $A:= H/J$ is a Hopf
algebra of dimension $4p$ and $H^{co \pi} = \Bbbk\{1,x \}$.
 By Proposition \ref{prop:R-skew}, $x^2 =0$  so that $K \cong \mathcal{A}_{4}^{\prime}$
 or $\mathcal{A}_{4,\xi}^{\prime \prime\prime}$ if $G(H) \cong C_4$
and $K \cong \mathcal{A}_{2,2}$ if $G(H) \cong C_2 \times C_2$.
Since $(\mathcal{A}_{4}^{\prime})^{*} \simeq \mathcal{A}_{4,\xi}^{\prime \prime\prime}$,
in the former case, $H^{*}\simeq R^\ast\# \mathcal{A}_{4,\xi}^{\prime \prime\prime}$ or
$ R^\ast\# \mathcal{A}_{4}^{\prime}$,
and
since $\mathcal{A}_{2,2}$ is self-dual,
we have $H^\ast \cong R^\ast \# \mathcal{A}_{2,2}$ in the later.

%

 \par   Since $\pi: H \rightarrow A:=H/J$ is
  injective on $\Bbbk G(H)$, $4$ divides $|G(A)|$.  Thus, by Theorem \ref{th: ChNg I},
  if $A$ is not semisimple,
   $A$ is pointed.
  But every pointed Hopf algebra of dimension $4p$ has group of grouplikes of order
  $2p$  which is not divisible by $4$,
   so $A$ must be
  semisimple.

\par Since $\mathcal{A}_{2,2}$ is self-dual,
we have $H^\ast \cong R^\ast \# \mathcal{A}_{2,2}$.
The same argument as for $H$ then gives us
  a Hopf algebra epimorphism from $H^\ast$ to a
semisimple Hopf algebra $B$ of dimension $4p$ with coinvariants
$\{1,y\}$ where $y$ is   $(1,h)$-primitive.

Then $B^\ast$ is isomorphic to a sub-Hopf algebra of $H$, call it $L$.
Since $L$ is cosemisimple, $L
\subseteq H_0$ and we wish to show equality.
Since $L$ has dimension $4p$, the sub-Hopf algebra $\langle L,x \rangle$ of $H$ generated by $L$
and $x$ is all of $H$.
Since $\pi(x) = 0$, this means that by dimensions $\pi$
 is injective on $L$ and so $\pi: L \cong A$ is a Hopf algebra isomorphism.  This
implies that $H \cong S \# A$ where $S = \Bbbk\{1,x \}$ is a braided Hopf algebra in
$_A^A\mathcal{YD}$ and thus $H$ has
the Chevalley property.  Reversing the roles of $H^\ast$ and $H$ in
the above argument we get that $H^\ast$ also has the Chevalley property.

 \par (ii) Now suppose that $|G(K)|= 2$ and $K \cong \mathcal{A}_2$.
Then $G(K) = \langle g \rangle$ and $K$ is generated by $g$ and
 two $(1,g)$-primitives, $x$ and $x^\prime$. Let $J$ be the Hopf ideal of $H$
generated by $x$  and as in (i),
 $J = Hx  $.  Thus as a $\Bbbk$-space,
  $J = Span \{r_ig^jz | r_1, \ldots, r_p \mbox{ a basis for }
R, j=0,1, z \in \{x, x^\prime x   \} \}$.  Thus   $\dim J \leq 4p$
 so that $\dim H/J \geq  4p$ and is a divisor of $8p$ so   $\dim H/J =4p$
and as above there is a Hopf algebra
epimorphism $\pi: H \rightarrow A$ where $A:= H/J$ is a Hopf algebra
of dimension $4p$ and $H^{co \pi} = \{1,x \}$.
Since $\pi(x^\prime), \pi(g)$ generate a sub-Hopf algebra of $A$ isomorphic to
$H_4$, then $A$ is not semisimple.  If $A^\ast$
is pointed, then $H^\ast$ has grouplikes of order $2p$.  This is a contradiction since
$H^\ast \cong R^\ast \# K^\ast$ with $K^\ast \cong \mathcal{A}_2$, so that
by Proposition \ref{pr:4-4-boson}(i), $G(H^\ast) = G(\mathcal{A}_2) \cong C_2$.
   Suppose that $A$ is pointed.  Since $A^\ast$ is not pointed, then
  $A \cong \mathcal{A}(-1,1)$ in the notation of
Section \ref{sect: 4p}. But this is impossible since $\mathcal{A}(-1,1)$ has no
sub-Hopf algebra isomorphic to $H_4$.    By
Theorem \ref{th: ChNg II}, for $p\leq 11$,  $A$ is either semisimple,
pointed or basic.
\end{proof}

\begin{cor}\label{cor: R comm ss}
Suppose $H \cong R \# K$ where $K$ is a pointed Hopf algebra of dimension $8$, and $R$ is commutative and
semisimple.
\par (i) If $|G(H)|=2$, then there is a Hopf algebra map $\pi$ from $H$ onto a Hopf algebra $A$ of dimension $4p$
which is nonsemisimple, nonpointed and nonbasic.
\par (ii) If $|G(H)| =4$ and furthermore $K$ is basic and $R^\ast$ is commutative and semisimple,
 then  $H$
and $H^\ast$ have the Chevalley property.
\end{cor}

\begin{proof}
It remains only to show that under the given conditions there is a $(1,g)$-primitive $x$ such that $x \cdot R =0$.
Since $R$ is semisimple commutative, $R$ can be written as the sum of one-dimensional simple ideals $Re_i$ with
$e_i$ a central primitive idempotent.  Now apply Lemma \ref{lem:x-R=0} and Proposition
\ref{pr: boson implies chevalley}.
\end{proof}

\begin{cor}\label{cor: type 4,4 dim 24}
If  $\dim H=24$ and $H$ has type $(4,4)$, then $H$ and $H^\ast$ have the Chevalley property.
\end{cor}

\begin{proof}
 By Corollary \ref{cor:24-4-4-1}, $H \cong R \#K$ where $K$, $K^\ast$
are pointed Hopf algebras of dimension $8$, $R$ is a
semisimple braided Hopf algebra in $^K_K\mathcal{YD}$
 of dimension $3$, and $R^\ast$ is a semisimple braided Hopf algebra in $^{K^\ast}_{K^\ast}\mathcal{YD}$ of dimension $3$.
 Since all simple representations of $R$ and $R^\ast$ must be one-dimensional, $R,R^\ast$ are commutative and
 the result follows from Corollary \ref{cor: R comm ss}.
\end{proof}

 \begin{remark} Suppose that $H$ is of type $(2^i,2^j)$,   has dimension $24$ and $H \cong R \# K$ where $K$ is pointed of dimension $8$.  Then $|G(H)| \neq 2$. For by Proposition \ref{pr:R-semisimple}, $R$ and $R^\ast$ are semisimple and thus, since both have
 dimension $3$, they are commutative also. Then the conditions of Proposition \ref{pr: boson implies chevalley} hold.
 \end{remark}

The next remark summarizes the results
proved for various particular dimensions $8p$, with $H$ nonsemsimple, nonpointed,
nonbasic as assumed throughout this section.
\begin{remark}\label{rm: summary}
\begin{enumerate}
\item[(i)] From Remark \ref{rm: 3 5 7}, if $\dim H = 24, 40, 56$, then $|G(H)| \neq 8$.
\item[(ii)] From Corollary \ref{cor: 3 5 C4}, if $\dim H = 24,40 $, then type $(2p,2)$ is impossible and
for type $(2p,4)$, $G(H^\ast) \cong C_4$.
\item[(iii)]  From Corollary \ref{cor:2p-4-4p}, if $p=3,7,11$, $|G(H)| = 2p$, $H^\ast$ does not have the Chevalley property, and  $H$ does not contain a copy of $\mathcal{A}(-1,1)$, i.e.,
we are in Case (iii) of Proposition \ref{prop:2p-exact}, then   $G(H^\ast) \ncong C_4$.
\item[(iv)] From Corollary \ref{cor: 24 type 6,4 corad} if $\dim H = 24$ and $H$ has type $(6,4)$ then if $H^\ast$ does not have the Chevalley property,
then $H$ has a sub-Hopf algebra isomorphic to $\mathcal{A}(-1,1)$, $\dim H^\ast_0 =20$ and as coalgebras, either $H^\ast  \cong \matha_4^{\prime \prime } \oplus \mathcal{M}^\ast(2, \Bbbk)^4$ or  $H^\ast  \cong \matha_4^{\prime \prime } \oplus \mathcal{M}^\ast(4, \Bbbk)$.
\item[(v)] By Corollary \ref{cor: type 4,4 dim 24}, if $\dim H = 24$ and $H$ does not have the Chevalley property, then $H$ is not of type $(4,4)$.
\end{enumerate}
\end{remark}

\subsection{Hopf algebras of dimension $24$}
In this subsection we specialize to the case of $p=3$, $\dim H = 24$.  Unless otherwise stated, throughout this section
$H$ will denote a   Hopf algebra without the Chevalley property.

\par Our first result is a general statement for all Hopf algebras of dimension $8p$ and will need the
 following remark  about
nonabelian groups of order $4p$.

\begin{remark}\label{rm: DF}
Suppose that $L$ is a nonabelian group of order $4p$, $p$ an odd
prime. Then unless $p=3$ and $L = \mathbb{A}_4$, $L$ has a normal
subgroup $N$ of order $p$. (This follows from the Sylow Theorems;
see, for example, \cite[p. 34]{L}.) Then there is a Hopf algebra map
from $\Bbbk L$ to $\Bbbk (L/N)$ where $L/N$ is a group of order $4$.
Dualizing we see that $\Bbbk^L$ contains a sub-Hopf algebra
isomorphic to a group algebra of dimension $4$ and thus $G(\Bbbk^L)$
is a group of order $4$.
\end{remark}

\begin{prop}
\label{pr: nontriv gplike} Let $H$ be a nonsemisimple, nonpointed nonbasic Hopf algebra with
$\dim H = 8p$ and suppose $H$ has a
simple subcoalgebra $D$ of dimension $4$ stable under the antipode.
Then $H$ has a nontrivial grouplike element of order $2$.
\end{prop}

\begin{proof}
  Let $\mathcal{H}$ denote the  sub-Hopf algebra of
$H$ generated by $D$. Then
  $\dim \mathcal{H} \neq 2,4, p$ and so $\dim \mathcal{H} = 8, 2p, 4p $ or $8p$.
  \par
If $\dim \mathcal{H} = 8$, then  by the classification of Hopf
algebas of dimension $8$,
 \cite{W}, \cite{stefan}, $\mathcal{H} \cong \Bbbk [C_2 \times
C_2] \oplus \mathcal{M}^\ast (2,\Bbbk)$ as coalgebras if
$\mathcal{H}$ is semisimple and $\mathcal{H} \cong H_4 \oplus
\mathcal{M}^\ast(2, \Bbbk)$ if $\mathcal{H}$ is basic. In either
case, $\mathcal{H}$,
 and thus $H$,
contains a grouplike element of order $2$.

If $\dim \mathcal{H} = 2p$, then by \cite{Ng3}, $\mathcal{H}$ is
semisimple, so that  $\mathcal{H}= \Bbbk^{\mathbb{D}_{p}}$ and
    has a
grouplike of order $2$.

Now suppose that $\dim \mathcal{H} = 4p$. By Proposition
\ref{prop:natale-stefan},
  $\mathcal{H} $ fits into a central exact sequence:
\begin{equation*}
\Bbbk^G \overset{i}{\hookrightarrow} \mathcal{H}
\overset{\pi}{\twoheadrightarrow} A
\end{equation*}
for a group $G$ and $A$ a nonsemisimple basic Hopf algebra. Then
$|G| \in \{1,2,4,p,2p,4p \}$. If $|G| = 1$, then $\mathcal{H}$ is
nonpointed nonsemisimple but has pointed dual, so by Subsection
\ref{sect: 4p}, $\mathcal{H} \cong \mathcal{A}(-1,1)^\ast \cong H_4
\oplus \mathcal{M}^\ast(2, \Bbbk)^{p-1}$ as coalgebras and
consequently has a grouplike element of order $2$. If $|G| \in
\{2,4,2p\}$, then $\Bbbk^G$ has also a grouplike element of order
$2$. If $|G|=p$, then $p$ divides $|G(\mathcal{H})|$ and  $|G(H)|$
so that by Proposition \ref{pr: dim not p,8p,4p}, $G(H) \cong C_{
2p}$ and $H$ has a grouplike of order $2$. If $|G| = 4p$, then
$\mathcal{H} = \Bbbk^G$ for $G$ a nonabelian group of order $4p$. By
Remark \ref{rm: DF},
 $\Bbbk^G$ has a group of grouplikes of
order $4$ unless $p=3$, $G=\mathbb{A}_4$ and the dimension of
$\mathcal{H}$ is $12$. But if $\mathcal{H} = \Bbbk^{\mathbb{A}_4}$
does
  not have a grouplike of order $2$, then
 as a coalgebra $\Bbbk^{\mathbb{A}_4}
 \cong \Bbbk C_3 \oplus \mathcal{M}^\ast (3,\Bbbk)$.
   But $\mathcal{H}$ has a simple subcoalgebra
of dimension $4$, so this case is impossible.

Finally, assume that $D$ generates $H$ so that as above,
we have an exact sequence
$\Bbbk^G \overset{i}{\hookrightarrow} H
\overset{\pi}{\twoheadrightarrow} A$
for a group $G$ and $A$ a nonsemisimple basic Hopf algebra.
Since $H$ is assumed to be nonbasic, then $|G| \neq 1$, and
since $H$ is nonsemsimple, $|G| \neq 8p$. The argument above shows
that if  $|G| \in \{ 2,4,p,2p   \}$, then $H$ has a grouplike
element of order $2$. If $|G|$ is $8$ or $4p$, then $A$ has
dimension $p$ or $2$ respectively and so must be semisimple.  This
would imply that $H$ is semisimple, a contradiction.
\end{proof}

\begin{lema}\label{lem:24-gr-2}
If $\dim H = 24$ then $H$ has a grouplike element of order $2$.
\end{lema}

\begin{proof}
By Proposition \ref{pr: dim not p,8p,4p},   $G(H)\ncong C_p = C_3$
so it suffices to show that $H$ has a nontrivial grouplike
element, i.e., that $H_0$ is not of the form $\Bbbk\cdot 1 \oplus E$
where $E$ is a sum of simple subcoalgebras of dimension greater than
$1$. Suppose that
\begin{equation*}
H_0 = \Bbbk\cdot 1 \oplus \oplus_{i=1}^t D_i \mbox{ where  } D_i
\cong \mathcal{M}^\ast (n_i,\Bbbk) \mbox{ and } n_j \leq n_{j+1}.
\end{equation*}
By Proposition
\ref{prop:biti-dasca}, $\dim H_0 \leq 15$ so that the possibilities for $H_0$ are
 $H_0 = \Bbbk \cdot 1 \oplus \M^{\ast}(2,\Bbbk)^s   $ with $s=1,2,3$,
  $H_0 = \Bbbk \cdot 1  \oplus \M^{\ast}(3,\Bbbk)  $ or
  $H_0 = \Bbbk \cdot 1 \oplus \M^{\ast}(2,\Bbbk) \oplus \M^{\ast}(3,\Bbbk) $. If $H$ has  a simple
  subcoalgebra of dimension $4$  stable under the antipode then  by Proposition \ref{pr:
nontriv gplike}, $H$ has
  a grouplike element of  order $2$.
If  $H_0 = \Bbbk \cdot 1  \oplus \M^{\ast}(3,\Bbbk)  $ then  Proposition
\ref{prop:biti-dasca} implies that $\dim H \geq 26$, a contradiction. Thus only the cases
$H_0 = \Bbbk \cdot 1 \oplus \M^{\ast}(2,\Bbbk)^s   $ with $s=2,3$ and $S(D_i) \neq D_i$ remain.

\smallbreak

\par Suppose that $H_0 = \Bbbk \cdot 1 \oplus \sum_{i}D_i  $ with $D_i \cong \M^{\ast}(2,\Bbbk)$
and $S(D_i) = D_j$ for some $j \neq i$. Note that $\dim H_0 >8$.
Let $\D  $ denote the set of $D_i$. Then since $4$ divides $\dim D_i$, $2\dim P^{ 1, \D}$, and $\dim P^{\D,\D}$, then
$4$ divides $1 + \dim P^{1,1}$ and $\dim P^{1,1} \geq 3$. Thus by Lemma \ref{lema:fukuda}, $P^{1,1}_\ell$
is nondegenerate for some $\ell >2$. Then $P_{m}^{1,D_i}, P_1^{D_i,1}, P_1^{1, S(D_i)}$ are nondegenerate for
$m = \ell -1\geq 2$, some $i$. Then $2\dim P^{1,\D} \geq 8$.  Since $P_1^{D_i,1}$ and $P_m^{S(D_i),1}$
are nondegenerate then $P^{D_i, \D}$ and $P^{S(D_i), \D}$ are nondegenerate and $\dim P^{\D,\D} \geq 8$.
But this is impossible if $\dim H = 24$.
\end{proof}

\begin{remark} Similar arguments to the proof of Lemma \ref{lem:24-gr-2}
apply if $\dim H=4n$
and $H_0 = \Bbbk \cdot 1 \oplus \sum_{i=1}^t D_i$ with $ D_i = \M^\ast(2,\Bbbk) $ and
$D_i \neq S(D_i)$ for all $i$.  Let $\D $  denote the set of $D_i$.
Then $2 \dim P^{1,\D} + \dim P^{\D,\D} \geq 20$ where $\D$ denotes the set of
simple $4$-dimensional subcoalgebras.
\par For,   we may suppose that $P^{1,1} = P^{1,1}_\ell$ with $\ell \geq 3$.
  Then $P_1^{1,C},P_{\ell -1}^{C,1}, P_1^{S(C),1},P_{\ell -1}^{C,E},
P_{\ell -2}^{C,D}, P_1^{D,1}$ are nondegenerate
 for some $C,D,E \in \D$
 so that $2\dim P^{1,\D} \geq 8$ and $\dim P^{\D,\D} \geq 8$.
  Furthermore, since $\ell -1\geq 2$, then $P_1^{C,X}, P_{\ell -2}^{X,E}$
are nondegenerate for some coalgebra $X$.
 If $\dim X = 1$,
then $P_1^{C,1}, P_1^{S(C),1}, P_{\ell -1}^{C,1}$ are nondegenerate and
$2 \dim P^{1,\D} \geq 12$. If $\dim X = 4$,
then $P_1^{C,X}, P_{\ell -1}^{C,E}, P_{\ell -2}^{C,D} $ are nondegenerate and the statement follows.
\end{remark}

We finish the section with the proof of Theorem \ref{thm:24}.

\bigbreak
\noindent {\bf Proof of Theorem B.}
Let $ \dim H =24$ and suppose that $H$ does not have the Chevalley property.
Then $ |G(H)|\neq 1, 3, 8, 12$ or $ 24 $, by Lemma
\ref{lem:24-gr-2},
Remark \ref{rm: 3 5 7}
and Proposition \ref{pr: dim not p,8p,4p}. Since $ |G(H)| $
divides $ \dim H $, we have that $|G(H)|\in  \{ 2, 4,6 \}$ and by Remark \ref{rm: summary},
the proof is complete.
\qed

\newpage
\section{Open cases}
The following table enumerates all   open cases in the
classification of Hopf algebras of dimension less than $100$ up to isomorphism.
In this table, $p$ is   arbitrary, not necessarily odd.

\begin{table}[here]
\begin{center}
\tiny{\begin{tabular}
{|p{1cm}|p{3.1cm}|p{3,9cm}|p{3,4cm}|p{3,4cm}|} \hline
{\bf $\dim H$} & {\bf Semisimple} & {\bf Pointed}&
 {\bf Chevalley} & {\bf Other}\\
\hline\hline
$ p$ & {\bf Completed:}
\newline All trivial \cite{Z}
 &  {\bf None} & {\bf None}
 & {\bf None:}  \cite{Z}
 \\ \hline

$ 2p$\newline
 $p$ odd & {\bf Completed:}
\newline All trivial \cite{ma-2p}\footnotemark
 &  {\bf None}
& {\bf None} & {\bf None:}   \cite{Ng3}
 \\ \hline

$ p^2$ & {\bf Completed:}
  All trivial \cite{masuoka-p^n}
 &  {\bf Completed:} $\exists$ $p-1$, the Taft Hopf algebras \cite{andrussch}
& {\bf None} & {\bf None:}  \cite{Ng}
 \\ \hline

$ pq$ & {\bf Completed:}
All trivial \newline \cite{ma-6-8, Ng3, EG, GW, So, pqq2}
 &  ${\bf None}$
& {\bf None} & {\bf None:}    for $p<q \leq 4p+11$  \cite{Ng4}
\newline {\bf Open:}
  $87$, $93$.
 \\

 \hline
$p^3$ & {\bf Completed:} \newline    $p=2 $ ,  $\exists$    $  1$ \cite{k-p} \cite{ma-6-8}
\newline $ p$ odd , $\exists$  $p+1$ \cite{ma-pp}
& {\bf Completed:} $p=2$, $\exists$ $5$ \cite{stefan}
\newline $p$ odd $\exists$ $(p-1)(p+9)/2$ \cite{AS2, CD, SV}
& {\bf {\bf None}}  &
${\bf None:}\newline $   $8$ \cite{W}, \cite{stefan}
\newline $27$ \cite{GG}, \cite{bg}
\\
\hline

$2p^2$
\newline $p$ odd& {\bf Completed:} $\exists$ $2$, they are duals
\cite{masuoka-further}, \cite{pqq}
& {\bf Completed:}\newline
$\exists$ $4(p-1)$    \cite[A.1]{andrunatale}&
{\bf None} &
{\bf None:}\footnotemark
 \cite{hilgemann-ng} \\
\hline

$pq^2$
\newline $p$ odd& {\bf Completed:}\footnotemark
 \newline  \cite{G, masuoka-further, pqq, pqq2, clspqq, eno-08}& {\bf Completed:}\newline
$\exists$  $4(q-1)$   \cite[A.1]{andrunatale}
& {\bf None:} \cite[Lemma A.2]{andrunatale}  & {\bf Completed:}
 $12$ \cite{natale}
\newline    20, 28, 44 \cite{ChNg}
\newline
{\bf Open:}
  $ 45$, $ 52$, $ 63$,
$ 68$, $ 75$, $ 76$,
$ 92$, $ 99$. \\
\hline

  $ pqr$  & {\bf Completed:}\footnotemark \newline
\cite{pqq, pqq2, eno-08} &  {\bf None} &
{\bf None: }   Prop. \ref{prop:no-chev-rpq} &
{\bf Completed:}  $30$ \cite{fukuda-30} \newline
{\bf Open:}
  $42$,
$66$,
$70$, $78$\\
\hline

$ p^4$ & {\bf Completed:} $p=2$, $\exists$ $16$ \cite[Theorem 1.2]{kashina} \newline {\bf Open:}  81
 & {\bf Completed:} $16$; $\exists$ $29$
\cite{pointed16} \footnotemark
\newline  {\bf Completed:} $p$ odd  \cite{AS2}.
Infinite nonisomorphic families exist \cite{AS2}, \cite{bdg}, \cite{Gel}\footnotemark[10]  &
{\bf Completed:} 16 \cite{de1tipo6chevalley}\newline $   \exists $ 2
selfdual, coradical $A_8$
\newline{\bf Open:} 81 & {\bf Completed:} 16
  \cite{GV}
\newline
{\bf Open:}
$81$
\\
\hline
$ p^3q$  & {\bf Open}
 & {\bf Completed:}\footnotemark \newline $24$, $40$, $54$, $56$ \cite{G1}
\newline {\bf Open:}  $88$ \footnotemark[9]
& {\bf Open}
 & {\bf Open:}\newline
  $24 $,
$40 $,
 $54 $,
$56 $,
   $88  $.

 \\
\hline$ p^2q^2$ & {\bf Open}
& {\bf Completed:} 36
\cite{G1}  \newline
 {\bf Open:}   $100$ \footnotemark[9]
& {\bf Open} & {\bf Open}:
  $36$,
$100$
\\
\hline $ p^2qr$ & {\bf Open} & {\bf Completed:}  $60$ \cite{G1}
\newline {\bf Open:}
  $84$, $90$ \footnotemark[9]
& {\bf Open} & {\bf Open:}
  $60$, $84$, $90$
\\
\hline
$ p^3q^{2}$ & {\bf Open}
& {\bf Open}\footnotemark   &
{\bf Open }\footnotemark
 & {\bf Open:}  $72$
 \\
\hline
$ p^n$ \newline $n=5,6$
& {\bf Open} &  {\bf Completed:}
 $32$. \cite{G3} Infinite families of nonisomorphic
Hopf algebras exist. \cite{G3}, \cite{b iso}\footnotemark[10]
 \newline {\bf Open:} 64 &
{\bf Open}  & {\bf Open:}  $32$, $64$ \\

\hline
$p^{4}q$ & {\bf Open} & {\bf Completed:} 48
  \cite{G1}
\newline {\bf Open:}   $80$  &
{\bf Open}  & {\bf Open:}
  $48$, $80$
\\

\hline
$ p^5q$ & {\bf Open}
& {\bf Open}\footnotemark[6]
& {\bf Open}  & {\bf Open:}
  $96$\\ \hline
\end{tabular}}
\end{center}
\caption{ Hopf algebras of dimension $\leq
100$}\label{tab-abiertas}
\end{table}

\footnotetext[1] {Dimension  $6$ was classified in \cite{ma-6-8}.}
\footnotetext[2]{The classification for dimension $18= 2(3^2)$
was completed  in  \cite{d-fukuda}.}
\footnotetext[3]{The complete classification of semisimple Hopf algebras of
  dimension $12= 3(2^2)$ is given in \cite{fukuda}.}
\footnotetext[4]{The complete classification of semisimple Hopf algebras of
  dimension $30$ and $42$ is given in \cite{Na-mm}.}
  \footnotetext[5]{The duals to these are explicitly constructed in \cite{biti2}.}
\footnotetext[6]{ Pointed
Hopf algebras $H$ with  $\frac{\dim H}{|G(H)|}< 32$ or
$\frac{\dim H}{|G(H)|} =p^{3}$ were classified in \cite{G1}. }
  \footnotetext[7]{Pointed Hopf algebras with nonabelian
grouplikes  known to exist by  \cite{AHS}   dimension $p^3p^2$ ,
 \cite{FG}   dimension $p^5q$.}
  \footnotetext[8]{Nonpointed Hopf algebras with
Chevalley property known to exist  \cite{AV1, AV2}.}
\footnotetext[9]{ $\dim  p^3q, p^2q^2, p^2qr$:  For dimensions $88$, $100$, $84$, $90$, the classification of the pointed Hopf algebras was completed
for those with coradical  a group algebra of order a power of 2  in \cite{nichols} and \cite{G1}.}
\footnotetext[10]{The families of nonisomorphic pointed Hopf algebras of dimension $81$ consist of quasi-isomorphic Hopf algebras \cite{masdefending} but the duals of the families of nonisomorphic pointed Hopf algebras of dimension $32$ give an infinite family of non-quasi-isomorphic
Hopf algebras \cite{eg}.}

The columns from left to right  describe
the classification of Hopf algebras which are semisimple,    pointed   nonsemisimple,
nonsemisimple nonpointed with the Chevalley property, etc.
We call a Hopf algebra \emph{trivial} if it is a group
algebra or the dual of a group algebra. For dimension $mn^2$, pointed Hopf algebras
always exist; just take $\Bbbk C_m \otimes T_q$ where $q$ is a primitive $n$th root of unity.

\par Note that by
\cite[Prop. 1.8]{andrunatale}, a Hopf algebra of square-free dimension cannot be pointed.
Also note if for every divisor $m$ of some dimension $n$   the only semisimple Hopf algebras of
dimension $m$ are the group algebras, then there are no Hopf algebras of dimension $n$ with the Chevalley property.
For example, this is why there are no nonpointed Hopf algebras of dimension $p^3$ with the Chevalley property.

\par Examples of nonpointed but basic Hopf algebras do exist. They are given
by duals of nontrivial liftings which are not Radford bosonizations. See for
example \cite{biti}.

In general, this table does not contain references to partial results for a particular dimension even
though the literature may contain some.  For example the general classification for dimension $24$ is listed only as Open.
Also when a general result has been proven, the table cites only that result.  For example, \cite{hilgemann-ng} is cited
for the result that all Hopf algebras of dimension $2p^2$, $p$ odd, are semisimple or pointed; the specific case of dimension
$18$ was proved in \cite{d-fukuda}.  We have attempted to include references to some specific cases in the footnotes but make no claim
that these are complete.



\end{document}